\documentclass[11pt]{article}
\usepackage{amssymb}

\usepackage{graphicx}
\usepackage{amsmath}
\usepackage{amsfonts}
\usepackage{amsopn}
\usepackage{amsthm}
\usepackage{latexsym}
\usepackage{ulem}
\usepackage{xcolor}
\usepackage{mathtools}

\usepackage{caption}
\usepackage{subcaption}
\usepackage{float}

\usepackage{enumerate}

\usepackage{bbm}

\newcounter{remark}[section]

\newcommand\mt{m_{\rm tip}}

\usepackage[subnum]{cases}

\newtheorem{theorem}{Theorem}
\newtheorem{lemma}{Lemma}[section]
\newtheorem{definition}{Definition}[section]

\newcommand{\bN}{\mathbb{N}}
\newcommand{\bR}{\mathbb{R}}

\title{Slow Passage through a Saddle-Node Bifurcation in Discrete Dynamical Systems}
\author{Jay Chu \footnote{
		Department of Mathematics, National Tsing Hua University, 
		No.~101, Sec. 2, Kuang-Fu Road, Hsinchu 300, Taiwan.
		 (Email address: ccchu@math.nthu.edu.tw).} \&
Jun-Jie Lin \footnote{
		Department of Mathematics, National Tsing Hua University, 
		No.~101, Sec. 2, Kuang-Fu Road, Hsinchu 300, Taiwan.
		 (Email address: s104021601@m104.nthu.edu.tw).} \&
          Je-Chiang Tsai \footnote{
	 	Department of Mathematics, National Tsing Hua University, 
		No.~101, Sec. 2, Kuang-Fu Road, Hsinchu 300, Taiwan, 
		National Center for Theoretical Sciences, 
		No.1, Sec. 4, Roosevelt Road, Taipei 106, Taiwan. 
		(Email address: tsaijc.math@gmail.com). }
}

\begin{document}
\maketitle

\begin{abstract}
We study a discrete non-autonomous  system 
whose autonomous counterpart (with the frozen bifurcation parameter) admits a saddle-node bifurcation, and in which the bifurcation parameter slowly changes in time
and is characterized by a sweep rate constant $\epsilon$.
The discrete system is more appropriate for modeling realistic systems since only time series data is available. 
We show that 
in contrast to its autonomous counterpart, 
when the time mesh size $\Delta t$ is less than the order $O(\epsilon)$,
there is a bifurcation delay as the bifurcation time-varying parameter is varied through the bifurcation point, 
and the delay is proportional to the two-thirds power of the sweep rate constant $\epsilon$.
This bifurcation delay is significant in various realistic systems since it allows one to take necessary action promptly before a  sudden collapse or shift to different states.
On the other hand, when the time mesh size $\Delta t$ is larger than the order $o(\epsilon)$, the dynamical behavior of the solution is dramatically changed before the bifurcation point. This behavior is not observed in the autonomous counterpart. 
Therefore, the dynamical behavior of the system strongly depends on the time mesh size.
Finally. due to the very discrete feature of the system, there are no efficient tools for the analytical study of the system.
Our approach is elementary and analytical. 
\end{abstract}


\section{Introduction}

It is widely known that many realistic systems undergo a critical transition, meaning that
the qualitative behavior of the system dramatically changes 
as some time-changing conditions move the system toward and cross a critical threshold termed a tipping point~\cite{Bury:2021, van:2016, Gladwell:2000}.
Examples include financial crisis, population collapse, and climate change~\cite{Ivashina:2010, Dai:2012, Lenton:2008}.
The significant qualitative change in system behavior initiated by critical transitions may cause severe damage to economics, environmental resources, and public health if no suitable responses are taken before the tipping point~\cite{van:2016}.
Therefore, it is necessary to understand how and when critical transitions occur. 



Although there is no well-accepted rigorous definitions of critical transitions and tipping points~\cite{van:2016},
we treat a critical transition and a tipping point as being synonymous with a bifurcation and a bifurcation point, respectively. 
From this point of view, a system undergoing critical transitions can be modeled by an autonomous dynamical system with a bifurcation.
Estimating a bifurcation point in a real-world system is crucial since it allows one to take necessary actions before its collapse or its sudden shift to a different state.
A bifurcation occurs when a system parameter varies through a critical value giving rise to a dramatical topological change of its dynamical behavior,
and the critical value corresponds to a tipping point in the phenomenon of critical transitions~\cite{Grziwotz:2023}.
The typical (one-dimensional) bifurcations are saddle-node (fold), transcritical, and pitchfork types~\cite{Strogatz, Perko:2006}. 
Due to the presence of symmetry in the transcritical and pitchfork bifurcations, these two types of bifurcations are not structurally stable~\cite{Berglund:2006},
and so are not easily to be observed in realistic systems. 
On the other hand, the saddle-node bifurcation provided a good understanding of a number of biological processes, 
as in the examples of the gene state switching in {\it E. coli}~\cite{Ge:2015} and population models~\cite{Malchow08}. 

In a realistic system, the controlled/bifurcation parameter may change as time evolves, 
for instance, the laser with a saturable absorber~\cite{Arimondo:1987} and the nutrient inflow of a lake~\cite{Ibelings:2007}.
In fact, there is a continuing interest in a class of nonautonomous dynamical systems 
where the controlled parameter is not fixed but changes slowly in time~\cite{Mandel:1987, Baer:1989, Jung:1990, Haberman:2000, Diminnie:2000, Li:2019, Premra:2019, Kalyakin:2022}.
Mathematically, for the case of saddle-node bifurcations in continuous dynamical systems, 
this can be illustrated by the following non-autonomous system
\begin{equation} \label{e:intro_temp10}
   \frac{d x}{d\tau} = -x^2 - p(\tau), 
\end{equation}
where $p(\tau) = \epsilon \tau$, and $\epsilon\in(0,1)$ is a fixed constant and termed as a sweep rate parameter by Strogatz~\cite{Strogatz:2021}. 
This immediately arises a question: Can the dynamical behavior of system~(\ref{e:intro_temp10}) around $\tau = 0$
be understood by its autonomous counterpart system~(\ref{e:intro_temp20}) around $p = 0$?
\begin{equation} \label{e:intro_temp20}
   \frac{d x}{d\tau} = -x^2 - p.
\end{equation}
Note that for each fixed $p < 0$, $\sqrt{-p}$ (resp. $-\sqrt{-p}$) is a globally stable (resp. an unstable) equilibrium of system~(\ref{e:intro_temp20}).
Further, as $p$ is increased from negative values, a solution of system~(\ref{e:intro_temp20}) tracks the moving fixed point  
$\sqrt{-p}$, and then goes to negative infinity once the bifurcation parameter $p$ is crossed the bifurcation point $0$ from below.
Intuitively, the solution $x(\tau)$ of system~(\ref{e:intro_temp10}) with initial value $x(\tau_0) > 0$ at $\tau_0 < 0$ would follow the scenario above
as the time $\tau$ is increased through the critical time $\tau=0$.
In contrast to this intuition, when viewed as a function of $p=\epsilon t$, the solution $x$ of system~(\ref{e:intro_temp10}) does not track the moving fixed point  $\sqrt{-p}$, but instead deviates from them by  a quantity of order $O(\epsilon^{1/3})$, up to some $p=p_0$ of order $O(\epsilon^{2/3})$,  and then stays positive and is of order $O(\epsilon^{1/3})$ over a region of width, $O(\epsilon^{2/3})$, centered around $p = 0$, and finally goes to negative infinity after leaving this region (see Berglund and Gentz~\cite{Berglund:2006}). 
Therefore, the transition of the dynamic behavior of system~(\ref{e:intro_temp10}) is delayed by the amount of order 
$O(\epsilon^{2/3})$ when the bifurcation parameter $p$ is varied slowly in the manner $p = \epsilon \tau$.
In other words, the system dynamics respond to the bifurcation scenario with a lag of order 
$O(\epsilon^{2/3})$, which is termed as a slow passage effect~\cite{Mandel:1987, Baer:1989, Jung:1990, Haberman:2000, Diminnie:2000, Li:2019, Premra:2019, Kalyakin:2022}. 

We would like to point out that this bifurcation delay is observed in various realistic systems.
For instance, the two-thirds power of the sweet rate was asymptotically derived and experimentally tested in a bistable semiconductor laser that is periodically switched between its two stable states (see Jung, Gray, and Roy~\cite{Jung:1990}),
and in the exchange of stability between the ordered nematic equilibria and the isotropic branch in the nematic liquid crystals (see Majumdar et al.~\cite{Majumdar:2013}).   
This delay is significant in various fields of science since it allows us to take some manners to reverse system behavior before the occurrence of runaway change. 

In this paper, we would like to study a discrete non-autonomous system with a slowly varying time-dependent bifurcation parameter
whose autonomous counterpart (with the frozen bifurcation parameter) admits a saddle-node bifurcation.
In fact, it is the discrete version of system~(\ref{e:intro_temp10}) using the Euler forward discretization scheme.
The discrete system is more appropriate for modeling realistic systems since only time series data is available~\cite{Grziwotz:2023}. 
Roughly speaking, our results here can be summarized concisely:
\begin{enumerate}
\item 
For the pair $(\epsilon, \Delta t)$
with $\epsilon/\Delta t = O(1)$, the solution of the discrete version of system~(\ref{e:intro_temp10}) exhibits the bifurcation delay, as the solution of (\ref{e:intro_temp10}) did.
\item
For the pair $(\epsilon, \Delta t)$
with $\epsilon/\Delta t = o(1)$, the solution the discrete version of system~(\ref{e:intro_temp10}) may oscillate around the stable manifold ($x=\sqrt{-p}$) and then goes to extinction before the bifurcation point.
\end{enumerate}
Here $O(1)$ is a bounded function of $\epsilon\in(0,1)$,
while $o(1)$ a function of $\epsilon\in(0,1)$ with $o(1)/\epsilon \to 0$ as $\epsilon\to 0^+$.
In particular, the dynamical behavior of the solution for 
case 2 cannot be observed in the autonomous counterpart, i.e., system (\ref{e:intro_temp10}). 
Therefore, in contrast to the continuous cases, the solution behavior depends on the ratio of the time mesh size to the sweet rate.
An important implication of these results is that the bifurcation point associated with system~(\ref{e:intro_temp20}) cannot be early-warning signals for critical transition associated with discrete systems~\cite{Scheffer:2009}. 
Due to the discrete feature of the system, there are no efficient tools, such as fast-slow theory and matching techniques, for the analytical study of the system.
Our approach is elementary and analytical. 

Finally, the remainder of this paper is organized as follows.
In section 2, the model and the main results are stated.
Section 3 and 4 are devoted to the solution behavior with 
$\epsilon/\Delta t = O(1)$ and $\epsilon/\Delta t = o(1)$,
respectively.
An application of our theory is given in Section 5.

\section{The model and main result}
\setcounter{equation}{0}

\subsection{The model}


Discretizing the equation~(\ref{e:intro_temp10}) based on the forward Euler scheme, we obtain the 
following discrete dynamical system for $\{X(m)\}$:

\begin{equation} \label{e:main_temp50}
     X(m+1) = X(m)  - \Delta \tau X(m)^2  - \epsilon\tau_m\Delta \tau, \quad m \in \bN \cup \{0 \}
\end{equation}
where the time mesh size $\Delta \tau$ is a fixed positive constant, and 
$$
   \tau_m = \tau_0 + m\Delta \tau.
$$
We will study the dynamical behavior of Eq.~(\ref{e:main_temp50}).
Using the scaling
\begin{equation}
     x(m) =  X(m), \quad 
     t_m =   \epsilon \tau_m, \quad
     \Delta t = \epsilon  \Delta\tau,
\end{equation}
the equation is converted into the following:
\begin{equation*}
     x(m+1) = x(m) -\frac{\Delta t}{\epsilon}x(m)^2-\frac{\Delta t}{\epsilon}t_m,  \; m \in \bN \cup \{0 \}.
\end{equation*}
Note that the scaling time variable $t_m$ is exactly the slowly varying bifurcation variable $p$. 

Now, consider the initial value problem
\begin{subnumcases} 
{\label{e:main}}
     x(m+1)=-\frac{\Delta t}{\epsilon}x(m)^2+x(m)-\frac{\Delta t}{\epsilon}t_m, \label{e:maina} &\\
     x(0)= x_0, \label{e:mainb} &
\end{subnumcases}
where $x_0 \in \bR$ satisfies
$$
    x_0-\sqrt{-t_0}=\alpha\epsilon
$$
for some $\alpha>0$.

\subsection{Tipping time and tipping point}

In this subsection, we will introduce the terminologies tipping time and tipping point, after which 
the dynamical behavior of the solution will be dramatically changed. 
To see this, we first give two lemmas.
The first one concerns the existence of the solution of problem~(\ref{e:main})
and its property, as stated below. 
The proof is direct and so omitted. 

\begin{lemma}\label{l:tipping}
  Let $\{x(m)\}_{m\in\bN}$ be the solution of problem~(\ref{e:main}).
Then $x(m)$ is defined for $m \in \{0\}\cup\bN$.
Moreover, if $x(m)$ takes on non-positive values at $M\in\bN$ with $t_M \ge 0$,
then $x(m)$ is decreasing for $m \ge M$ and tends to negative infinity as $m \to \infty$.
\end{lemma}

\medskip
Therefore, the solution $\{x(m)\}_{m\in\bN}$ always takes negatives and never returns to positive values after it passes through the point $x_M$ with $t_M\geq 0$.
On the other hand, if the solution $\{x(m)\}_{m\in\bN}$ takes on non-positive values at $M\in\bN$ with $t_M < 0$,
then the solution $\{x(m)\}_{m\in\bN}$ may not be decreasing for $m \ge M$. 
However, if the solution $\{x(m)\}_{m\in\bN}$ lies below the unstable manifold for some $M\in\bN$ with $t_M < 0$, then it must be decreasing for $m \ge M$, as shown in the following lemma.
\begin{lemma}\label{l:earlytipping}
    Let $\{ x(m) \}_{m\in\mathbb{N}}$ be the solution of problem~(\ref{e:main}).
    Suppose there exists a $M\in\bN$ with $t_M<0$ such that $x(M)<-\sqrt{-t_M}$.
    Then $\{ x(m) \}_{m\in\mathbb{N}}$ is decreasing for $m\geq M$. 
    Moreover, $x(m)<-\sqrt{-t_m}$ for $m \ge M$ with $t_m < 0$.
\end{lemma}
\begin{proof}
    To proceed,  equation(\ref{e:maina}) can be rearranged as follows:
\begin{equation}\label{e:maina_convert}
        x(m+1) -  x(m)
        = -\frac{\Delta t}{\epsilon}(x(m)+\sqrt{-t_m})(x(m)-\sqrt{-t_m}).
    \end{equation}
   The first assertion can be proved by induction on $m$. Indeed, when $m=M$, 
    since $t_M<0$ and $x(M)<-\sqrt{-t_M}$, 
    it follows from equation~(\ref{e:maina_convert}) that $x(M+1)-x(M)<0$, and hence that the first assertion holds for $m=M$.
    Assume the first assertion holds for $m=k$, that is, $x(k+1)<x(k)$. By equation~(\ref{e:maina_convert}) and $x(k)<0$, we have $x(k)+\sqrt{-t_k}<0$. Now, for $m=k+1$, we need to divide the discussion into two cases:\\
    Case 1: $t_{k+1}<0$. Since $x(k+1)<x(k)$ and $\sqrt{-t_{k+1}}<\sqrt{-t_{k}}$, $x(k+1)+\sqrt{-t_{k+1}}<x(k)+\sqrt{-t_{k}}<0$. 
    As a result, by equation~(\ref{e:maina_convert}),
    \begin{eqnarray*}
        x(k+2)
        &=& x(k+1)-\frac{\Delta t}{\epsilon}(x(k+1)+\sqrt{-t_{k+1}})(x(k+1)-\sqrt{-t_{k+1}})\\
        &<& x(k+1)-\frac{\Delta t}{\epsilon}(x(k)+\sqrt{-t_{k}})(x(k+1)-\sqrt{-t_{k+1}})\\
        &&\big(\text{since 
        $x(k+1)<0$\big)}\\
        &<& x(k+1) \hspace{0.2cm} 
        \big(\text{since $x(k)+\sqrt{-t_{k}}<0$}\big),
    \end{eqnarray*}
    and so the first assertion holds for $m=k+1$.\\
    Case 2: $t_{k+1}>0$. Equation~(\ref{e:maina}) gives
    \begin{eqnarray*}
        x(k+2)
        &=& x(k+1)-\frac{\Delta t}{\epsilon}(x(k+1)^2+t_{k+1})\\
        &<& x(k+1)  \hspace{0.2cm} \big(\text{since $t_{k+1}>0$}\big),
    \end{eqnarray*}
    and so the first assertion holds for $m=k+1$.\\
    Therefore, by induction, $\{ x(m) \}$ is decreasing for $m\geq M$.
    Finally, the second assertion follows from the first assertion and equation~(\ref{e:maina_convert}). 
\end{proof}
Motivated by Lemma~\ref{l:tipping} and Lemma~\ref{l:earlytipping}, we introduce the tipping time and tipping points, as stated in the following definition.   

\begin{definition}
   Let $\{x(m)\}_{m\in\bN}$ be the solution of problem~(\ref{e:main})
and 
$$
  \mt := \inf\{m \in\bN:\, x(m) < -\sqrt{\max\{ -t_m,0 \}} \}.
$$
Then the corresponding time $t_{\mt}$ is said to be a tipping time of  $\{x(m)\}_{m\in\bN}$ and the corresponding point $x(\mt)$ is said to be a tipping point of  $\{x(m)\}_{m\in\bN}$.
In this case, the solution $\{x(m)\}_{m\in\bN}$ is said to be tipping at the tipping time $t_{\mt}$.
\end{definition}

Lemma~\ref{l:tipping} and Lemma~\ref{l:earlytipping}
indicate that once the solution $\{x(m)\}_{m\in\bN}$ passes through the tipping point, then it becomes decreasing and tends to negative infinity as $m \to \infty$.
Therefore, the dynamical behavior of the solution $\{x(m)\}_{m\in\bN}$ is dramatically changed after it passes through the tipping point.



\subsection{Main analytical results}

\noindent
To formulate our first main result, we need the following notation. 
\begin{definition}
Let $x_1(m,\epsilon), x_2(m,\epsilon)$ be two sequences defined for $m\in \bN \bigcup \{0 \}$ and parameter $\epsilon\in (0, \epsilon_0]$. 
Denote by
$$
    x_1(m,\epsilon) \asymp x_2(m,\epsilon),
$$
if there exists two constants $c_\pm >0$ such that
$$
    c_-x_2(m,\epsilon) \leq x_1(m, \epsilon) \leq c_+x_2(m,\epsilon)
$$
for all $m\in \bN \bigcup \{0 \}$ and $\epsilon\in (0, \epsilon_0]$.
\end{definition}

\medskip

\begin{figure}[ht]
    \centering
    \begin{subfigure}[b]{0.49\textwidth}
         \centering
         \includegraphics[width=\textwidth]{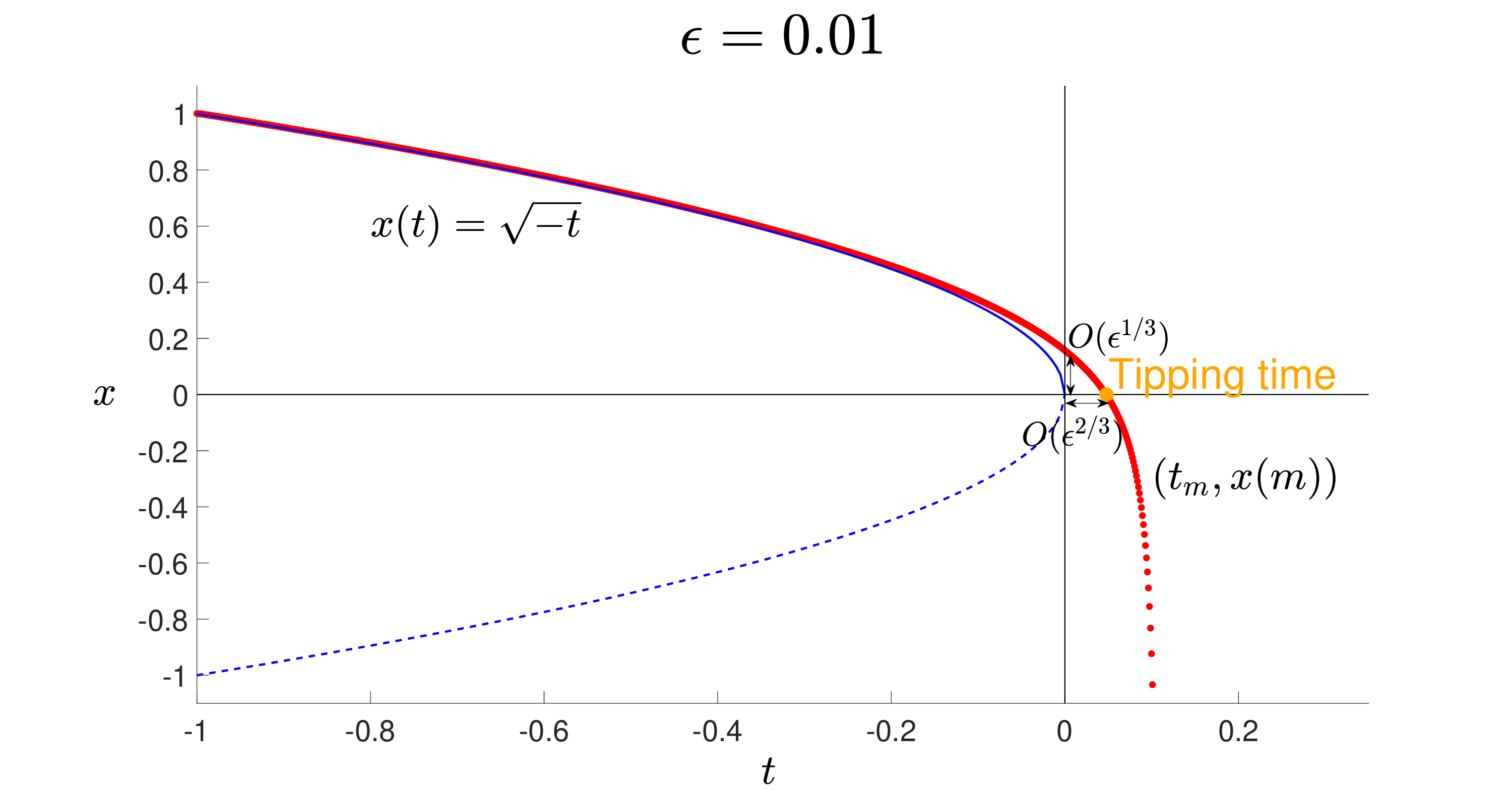}
         \caption{}
     \end{subfigure}
     \hfill
     \begin{subfigure}[b]{0.49\textwidth}
         \centering
         \includegraphics[width=\textwidth]{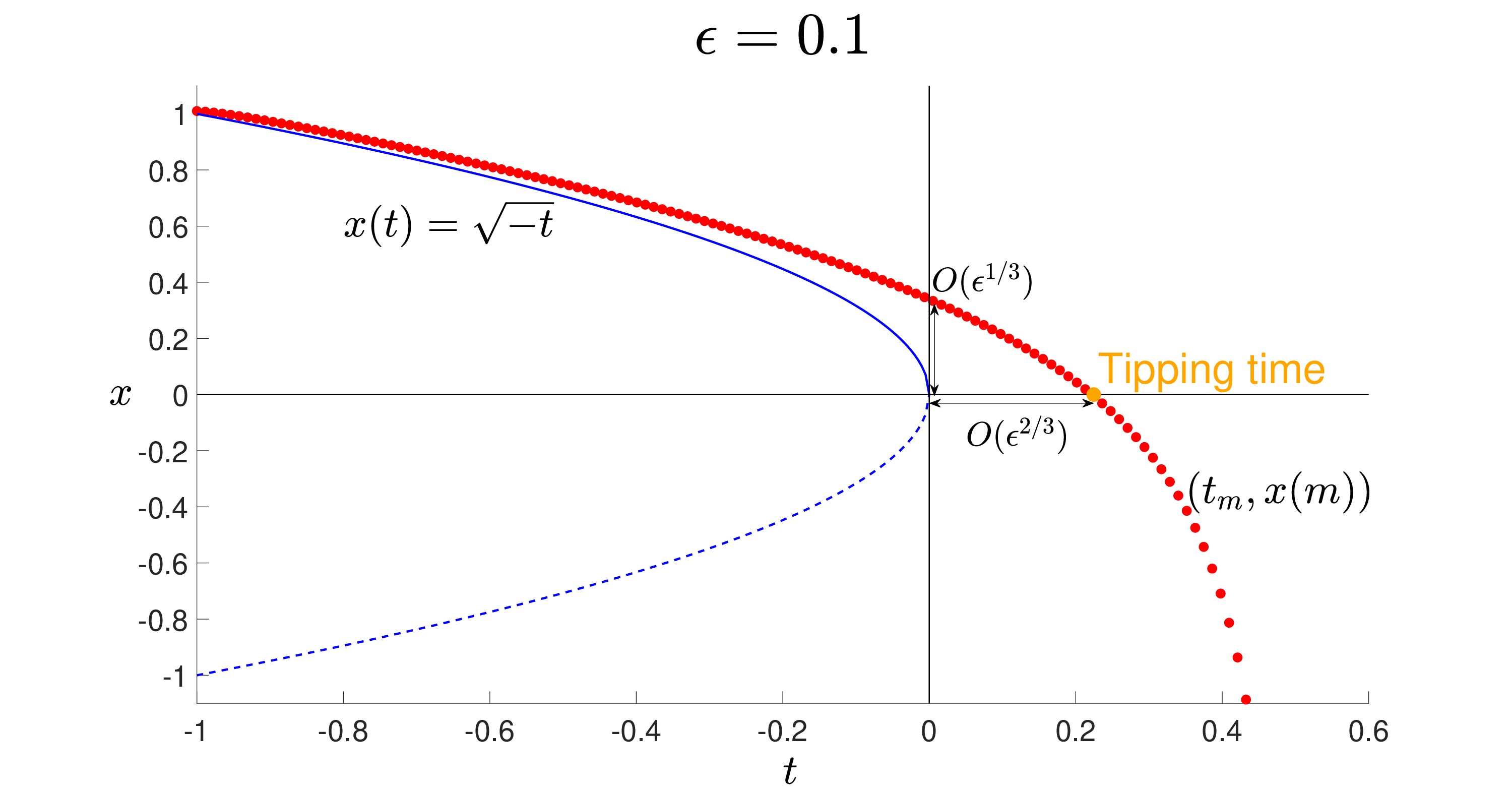}
         \caption{}
     \end{subfigure}
     \begin{subfigure}[b]{0.49\textwidth}
         \centering
         \includegraphics[width=\textwidth]{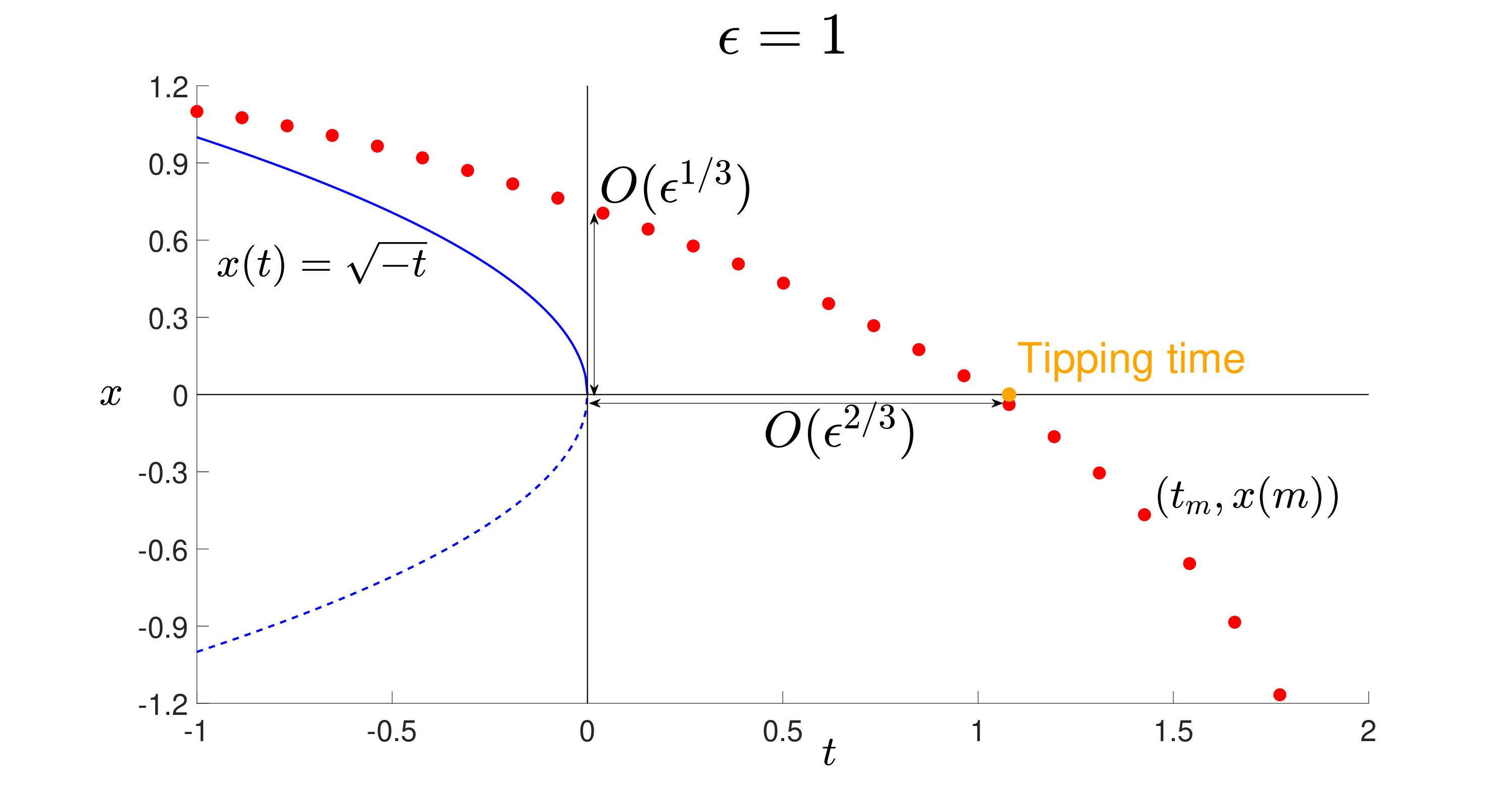}
         \caption{}
     \end{subfigure}
     \hfill
     \begin{subfigure}[b]{0.49\textwidth}
         \centering
         \includegraphics[width=\textwidth]{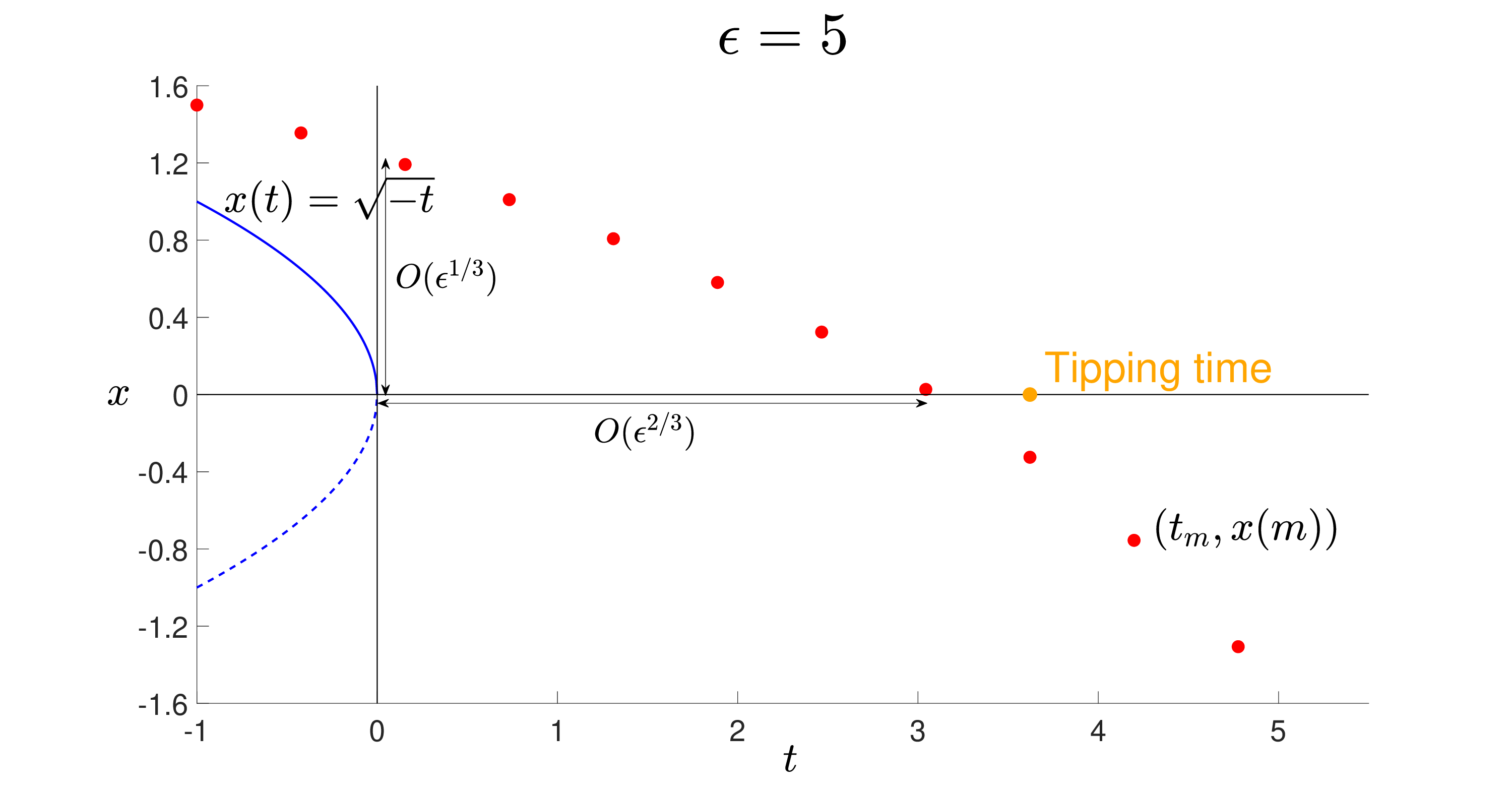}
         \caption{}
     \end{subfigure}
    \caption{\label{fig:1}
Solutions of problem~(\ref{e:main}) plotted in the $(t, x)$ plane for various $\epsilon \in (0, 1)$. The time variable $t$ is exactly the slowly varying bifurcation parameter $p$ due to the relation $p=\epsilon \tau$.
The stable (resp. unstable) equilibrium solution of the autonomous counterpart is plotted as a thick (resp. dashed) blue curve.
(a). $\epsilon = 0.01$,   (b). $\epsilon = 0.1$,  (c). $\epsilon = 1$, and  (d). $\epsilon = 5$ with $\Delta t=\epsilon\ln(2)/6$ for all cases.}
\end{figure}

\medskip

\noindent
Here is the result for the case with $\epsilon/\Delta t = O(1)$.
\begin{theorem}[{\bf Solution behavior with $\epsilon/\Delta t = O(1)$}] \label{t:main1}
Let $\{x(m)\}_{m\in\bN}$ be the solution of problem~(\ref{e:main})
and set 
\begin{equation*}
    \delta_0 = \min\Big\{ \frac{\sqrt{-t_0}}{2},  \frac{1}{3\sqrt{-t_0}}, \frac{\ln 2}{6} \Big\}.
\end{equation*}
Then there exists a small $\epsilon_0>0$ such that
for each pair $(\Delta t, \epsilon)$ with $\Delta t/ \epsilon \in (0, \delta_0)$ and $\epsilon \in (0, \epsilon_0)$,
two positive constants $c_1,c_2$ can be chosen so that
\begin{eqnarray}
x(m)-\sqrt{-t_m} &\asymp& \frac{\epsilon}{|t_m|} 
\hspace{0.5 cm} \text{for } m\in\bN \text{ such that }\;  t_0\leq t_m\leq -c_1\epsilon^{\frac{2}{3}}, \label{t:main1_temp30}\\
  x(m)&\asymp& \epsilon^{\frac{1}{3}}   \hspace{0.8 cm}\text{for } m\in\bN \text{ such that } -c_1\epsilon^{\frac{2}{3}}\leq t_m\leq c_2\epsilon^{\frac{2}{3}}. \label{t:main1_temp60}
\end{eqnarray}
Moreover, $x(m)$ is negative for $t_m > c_2\epsilon^{2/3}$ and approaches negative infinity as $m\to\infty$.
See Figure~\ref{fig:1} for an illustration.
\end{theorem}


\medskip

\noindent
Here is the result for the case with $\epsilon/\Delta t = o(1)$.
\begin{theorem}[{\bf Solution behavior with $\epsilon/\Delta t = o(1)$}] \label{t:main2}
Let $\{x(m)\}_{m\in\bN}$ be the solution of problem~(\ref{e:main}).
\begin{enumerate}[\rm(a)]
\item{\rm(Tipping at $m=1$)}
 Let $\Delta t>1/\alpha$. 
     Then $x(m)<-\sqrt{-t_m}$ at $m=1$,
     that is, $x(1)$ is the tipping point.
\item{\rm(Tipping at $m=3$)}
Let $\Delta t=C\epsilon^b$ with $C > 0$ and $b\in(0,1/2)$.
    Suppose that $\alpha>1/(-4t_0)$.
    Then there exists a small $\epsilon_0>0$ such that for $\epsilon\in(0,\epsilon_0)$, the solution $\{x(m)\}$ satisfies 
    \begin{enumerate}[\rm(i)]
        \item 
        $x(m) > 0$ and $(-1)^m(x(m) - \sqrt{-t_m})>0$ for $m=1, 2$; and
        \item 
        $x(m)<-\sqrt{-t_m}$ at $m=3$.
    \end{enumerate}     
\end{enumerate}
See Fig.~\ref{fig:Negative_Tipping} for an illustration.
\end{theorem}

The dynamical behavior of the solution given in Theorem~\ref{t:main2}
cannot be observed in the autonomous counterpart, i.e., system (\ref{e:intro_temp10}). 
Therefore, in contrast to the continuous cases, the solution behavior depends on the ratio of the time mesh size to the sweet rate.
Further investigation of the solution of problem~(\ref{e:main}) with negative tipping time is given in Section~\ref{sec:negative tipping time}.

\medskip

\begin{figure}[b!]
    \centering
    \begin{subfigure}[b]{0.42\textwidth}
         \centering
         \includegraphics[width=\textwidth]{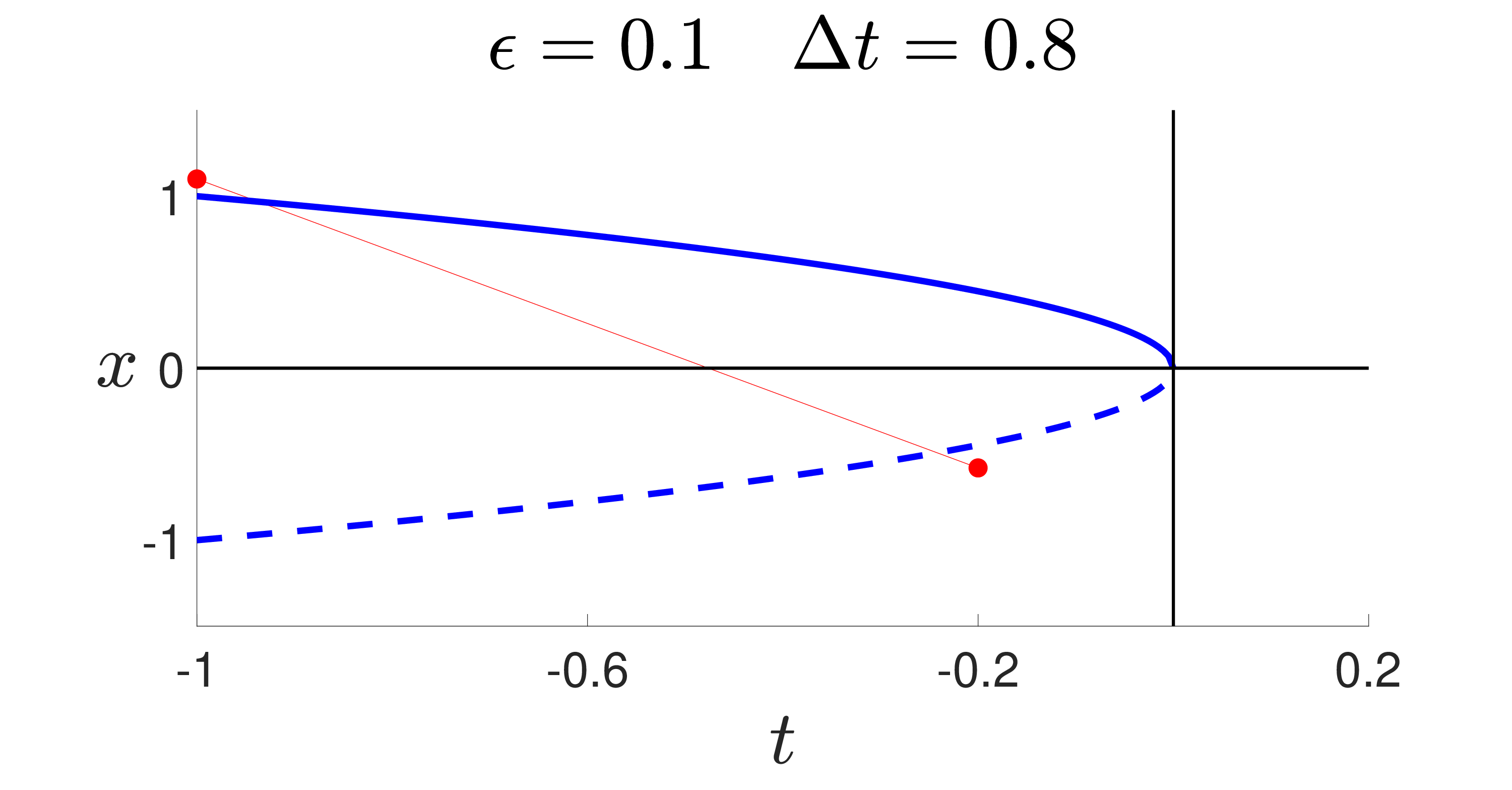}
         \caption{}
     \end{subfigure}
     \hfill
     \begin{subfigure}[b]{0.42\textwidth}
         \centering
         \includegraphics[width=\textwidth]{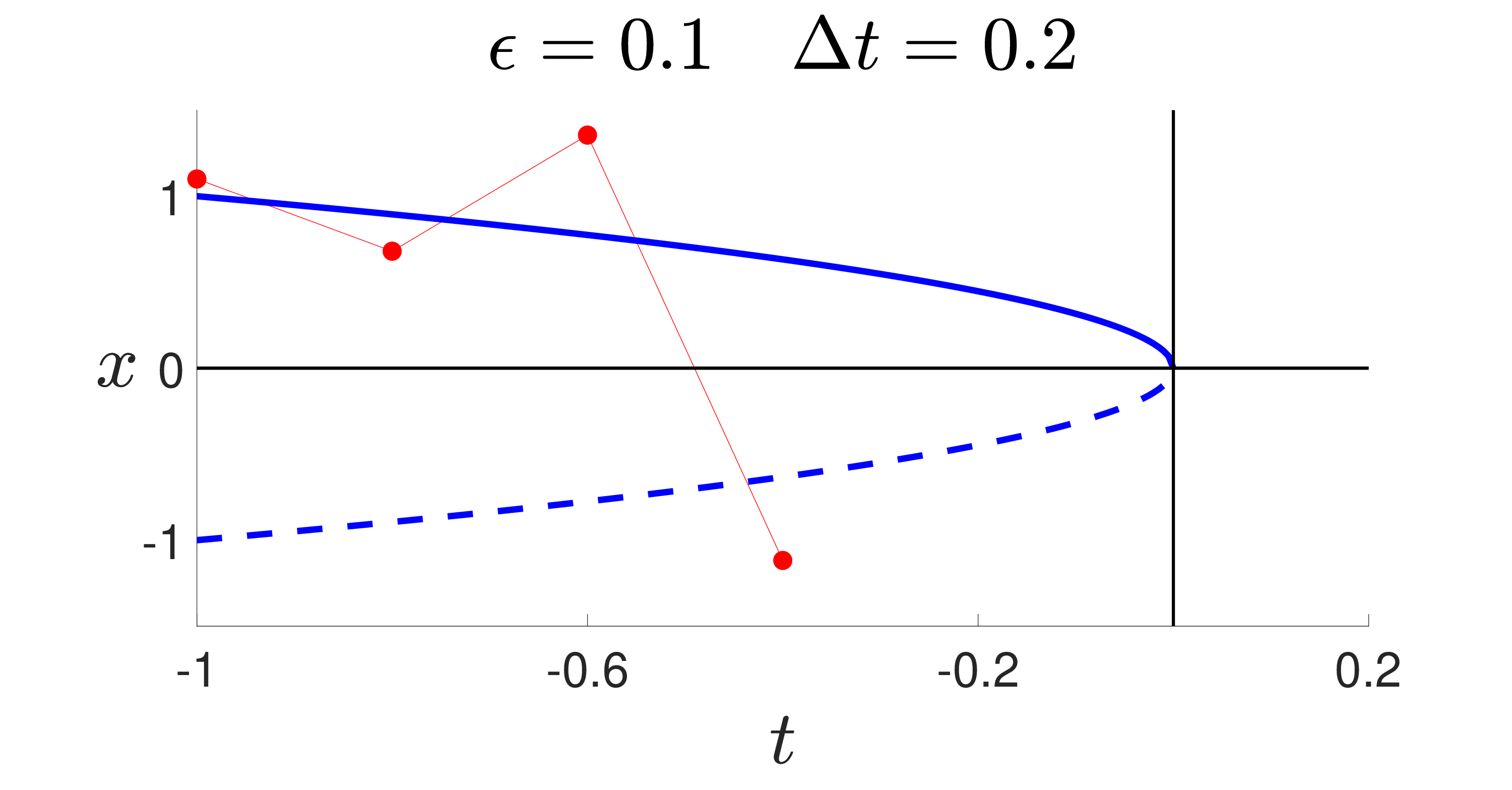}
         \caption{}
     \end{subfigure}
    \caption{\label{fig:Negative_Tipping}
Solutions of problem~(\ref{e:main}) (with negative tipping time) plotted in the $(t, x)$ plane:
(a). $\mt = 1$,   (b). $\mt = 3$.}
\end{figure}
\section{Solution behavior when $\epsilon/\Delta t = O(1)$}
\setcounter{equation}{0}

\subsection{Auxiliary lemma}

For the proof of Theorem \ref{t:main1}, we need the following comparison lemma.
\begin{lemma}[{\rm Comparison principle}]\label{l:compasrison}
    Let $a, g: \bN\cup\{0\} \to \bR$ be functions and $a$ is nonnegative.
    Suppose $x: \bN\cup\{0\} \to \bR$ and $\Tilde{x}: \bN\cup\{0\} \to \bR$
    are solutions of the problems
    $$
    \begin{cases}
        x(m+1) = a(m)x(m)+g(m)&\\
        x(0)=x_0,&
    \end{cases}
    $$
    and 
    $$
    \begin{cases}{}
        \Tilde{x}(m+1) \leq  a(m)\Tilde{x}(m)+g(m)&\\
        \Tilde{x}(0)=x_0,&
    \end{cases}
    $$
    respectively. 
    Then
    \begin{equation*}\label{e:lemma1}
        \Tilde{x}(m) \leq x(m) \quad \forall m \in \bN \cup \{0 \}.
    \end{equation*}
    Similarly, $x(m) \leq \Tilde{x}(m)$ for all $m \in \bN \cup \{0 \}$,
    provided if the sign $\leq$ is replaced by $\geq$ in the above two difference inequalities for $\tilde{x}$.
\end{lemma}
\begin{proof}
    We prove this lemma by induction on $m$. 
    When $m=0$, the assertion holds since $x(0) = \Tilde{x}(0) = x_0$.
    Assume that when $m=k$, the assertion holds, that is,  
    $\Tilde{x}(k)\leq x(k)$. Now, for $m=k+1$,
    \begin{eqnarray*}
        x(k+1) &=&     a(k)x(k)+g(k)\\ [1ex]
               &\geq&  a(k)\Tilde{x}(k)+g(k) 
                \quad \text{(since $x(k) \geq \tilde{x}(k)$ and $a(k) \geq 0$)} \\ [1ex]
               &\geq&  \Tilde{x}(k+1)
    \end{eqnarray*}
    the assertion holds for $m=k+1$.
    Therefore, by induction, the assertion of this lemma is established.
\end{proof}

The proof of Theorem~\ref{t:main1} consists of two parts:
one for (\ref{t:main1_temp30}), and the other for (\ref{t:main1_temp60}).
For the proof, the domain of the solution are divided into three sub-regions:
Outer region, corner layer, and beyond-corner layer (see Fig.~\ref{e:region_proof}).

\begin{figure}[ht]
    \centering
    \includegraphics[width=\textwidth]{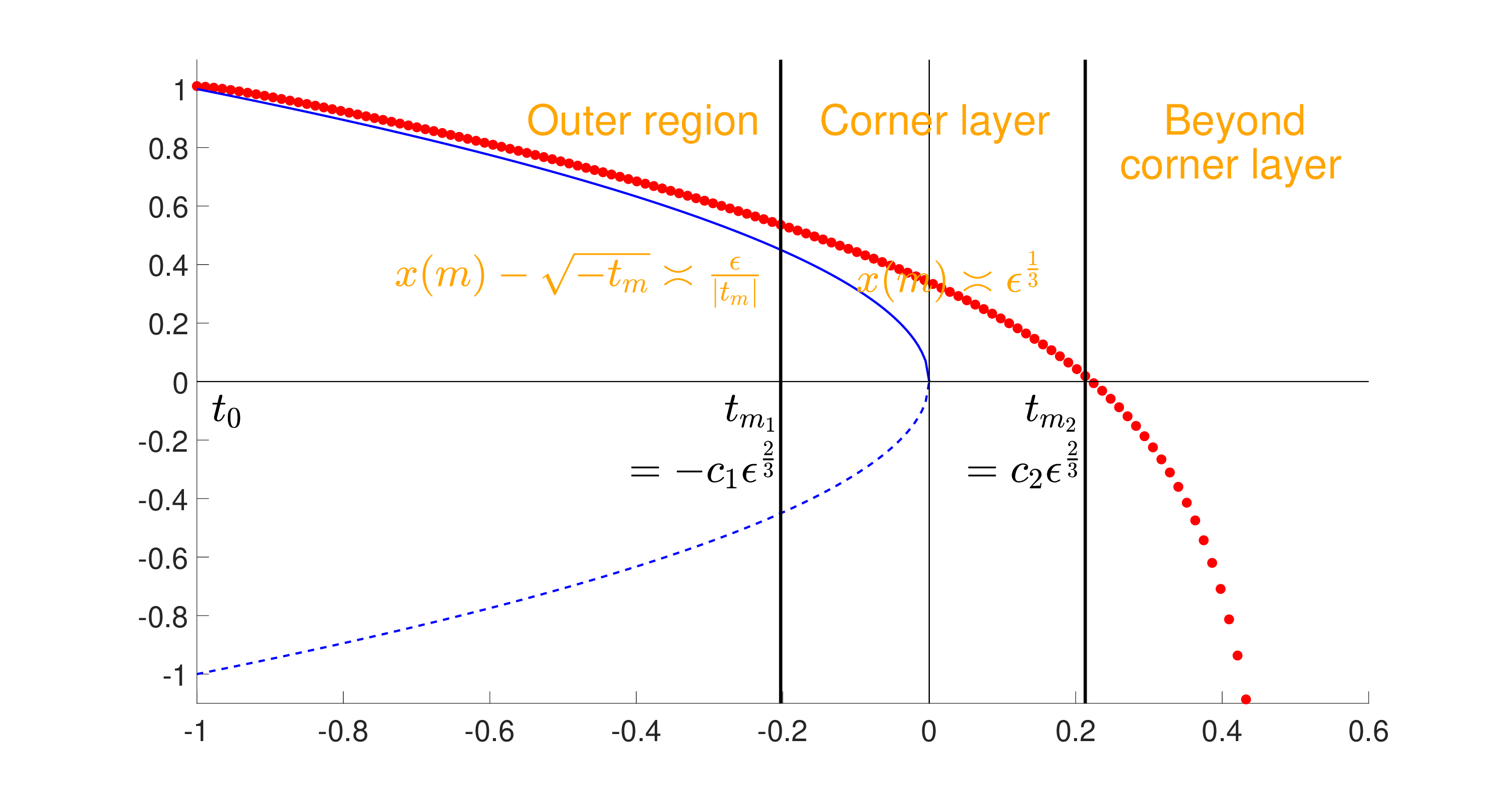}
    \caption{Outer region, corner layer, and beyond corner layer for the proof of Theorem~\ref{t:main1}}
    \label{e:region_proof}
\end{figure}

\subsection{Outer region: Proof of (\ref{t:main1_temp30})}

In this subsection, we establish (\ref{t:main1_temp30}).
To facilitate the analysis, we set up the following notations:
\begin{subnumcases} 
{\label{e:notationtemp20}}
    m_0:=\sup\big\{m\in\mathbb{N}:\, t_m \le 0 \big\},  \label{e:notationtemp20a} &\\ 
    K =  \big(\alpha |t_0|+\frac{1}{4}\big), \; c_1 = K^{2/3}, \label{e:notationtemp20b} &\\ 
    m_1:=\sup\big\{m\in\mathbb{N}:\, t_m \le -c_1\epsilon^{2/3} \big\},  \label{e:main_m*} & 
\end{subnumcases}
and make the change of variable
$$
      z(m) = x(m) - \sqrt{-t_m}, \; m\leq m_0 .
$$
Then problem~(\ref{e:main}) becomes the following problem
\begin{subnumcases} 
{\label{e:main_convert}}
   z(m+1) = -\frac{\Delta t}{\epsilon}z(m)^2-\frac{2\Delta t}{\epsilon}z(m)\sqrt{-t_m}+z(m)+\frac{\Delta t}{\sqrt{-t_m}+\sqrt{-t_{m+1}}}, &\label{e:main_converta}\\ 
     z(0) = \alpha\epsilon. \label{e:main_convertb} &
\end{subnumcases}

\noindent
{\bf Upper bound of $z(m)$}.

We first give an upper bound of $z(m)$.
Let $A(m)=-\frac{2\Delta t}{\epsilon}\sqrt{-t_m}+1$ and $g(m)=\frac{\Delta t}{2\sqrt{-t_{m+1}}}$. Indeed, noting that $t_m$ is increasing in $m\in\mathbb{N}$,
the solution of problem~(\ref{e:main_convert}) satisfies the following  problem
\begin{subnumcases} 
{\label{e:main_convert2}}
     z(m+1) \leq 
      A(m)z(m)+g(m) \;\;\,forall m \le m_0, \nonumber &\\
     z(0) = \alpha\epsilon, \label{e:main_convert2b} &
\end{subnumcases}
where $m_0$ is defined by (\ref{e:notationtemp20a})
Now, we will bound $z(m)$ by the solution of the associated linear problem. To see this, 
it is known \cite{Anosov:2006} that the solution of the linear problem
$$
z(m+1)=A(m)z(m)+g(m), \quad
z(0) = \alpha\epsilon
$$ 
is given by
\begin{equation}\label{t:main1_proof}
    z(m) = \pi(m,0)z(0)+\sum_{j=0}^{m-1}\pi(m,j+1)g(j)
\end{equation}
where
$$
    \pi(m,j) = \prod_{i=j}^{m-1}A(i).
$$
By the choice of $\Delta t$, we have
$\Delta t/\epsilon \leq 1/3\sqrt{-t_0}$,
and so
$$
    0\leq A(m)=(-\frac{2\Delta t}{\epsilon}\sqrt{-t_m}+1)\leq 1 \quad \forall\, 0\leq m\leq m_0.
$$
Now, by applying the comparison lemma (Lemma \ref{l:compasrison})
to problems (\ref{e:main_convert}) and (\ref{e:main_convert2}),
it follows that
\begin{equation}\label{ineq:main1_proof1}
    z(m+1)\leq \pi(m,0)z(0)+\sum_{j=0}^{m-1}\pi(m,j+1)g(j)
    \;\;, \forall\, m \leq m_0.
\end{equation}

Now, we will estimate each terms in the right-hand side of (\ref{ineq:main1_proof1}).
For the first term $\pi(m,0)z(0)$, we recall that
$0 \le A(m) \le 1$ for $m \le m_0$.
Then $\pi(m,0)z(0)$ is estimated as follows:
$$
    \pi(m,0)z(0)\leq z(0) = \alpha\epsilon\leq \alpha|t_0|\frac{\epsilon}{|t_m|} \quad \forall \; m \leq m_0,
$$
where in the last inequality, we have used the fact that 
$|t_m|$ is decreasing in $m \le m_0$.
The second term in the right-hand side of (\ref{ineq:main1_proof1}) is bounded by
\begin{align*}
    \sum_{j=0}^{m-1}\pi(m,j+1)g(j)
    &= \sum_{j=0}^{m-1}\prod_{i=j+1}^{m-1}(-\frac{2\Delta t}{\epsilon}\sqrt{-t_i}+1)\frac{\Delta t}{2\sqrt{-t_{j+1}}}\\
    &\leq \frac{\Delta t}{2\sqrt{-t_m}}\sum_{j=0}^{m-1}\prod_{i=j+1}^{m-1}(-\frac{2\Delta t}{\epsilon}\sqrt{-t_{m}}+1)\\
    & \hspace{0.4 cm} \text{(since  $|t_i| \ge |t_m|$ for $i\le m$)}\\
    &= \frac{\Delta t}{2\sqrt{-t_m}}\sum_{j=0}^{m-1}(-\frac{2\Delta t}{\epsilon}\sqrt{-t_{m}}+1)^{m-1-j}\\
    &\leq \frac{\Delta t}{2\sqrt{-t_m}}\sum_{j=0}^{\infty}(-\frac{2\Delta t}{\epsilon}\sqrt{-t_{m}}+1)^j\\
    &= \frac{1}{4}\frac{\epsilon}{|t_m|} \quad \forall \, m \leq m_0.
\end{align*}
Finally, recall the $K$ is defined by (\ref{e:notationtemp20b}).
Then
a combination of the estimates above gives
\begin{equation}\label{ineq:main1_upperbound}
    z(m) \leq 
   \big(\alpha |t_0|+\frac{1}{4}\big)\frac{\epsilon}{|t_m|}
   = K\frac{\epsilon}{|t_m|}
    \quad \forall \, m \leq m_0,
\end{equation}
which in turn gives an upper bound of the solution $z(m)$ of problem (\ref{e:main_convert}).

\medskip

\noindent
{\bf Nonnegativeness of $z(m)$}.

Next, we will show that $z(m)$
is nonnegative.
Indeed, recall $c_1$ and $m_1$ are defined by  (\ref{e:notationtemp20b}) and (\ref{e:main_m*}).
Then we have the following lemma.
\begin{lemma}\label{l:main1}
   Suppose that 
   $\Delta t/\epsilon \leq 1/(3\sqrt{-t_0})$.
    Then 
$$
    0 \leq z(m)\leq  \sqrt{-t_m} \quad \forall m\leq m_1.
$$
\end{lemma}
\begin{proof}
    First, we prove the rightmost inequality.
    Indeed, by inequality (\ref{ineq:main1_upperbound}) and the choice of $m_1$ and $c_1$,
    $$
        z(m) \leq K\frac{\epsilon}{|t_m|}
        = \big(K\frac{\epsilon}{\sqrt{|t_m|^3}}\big)\sqrt{-t_m}
        \leq \big(K\frac{\epsilon}{c_1^{3/2}\epsilon}\big)\sqrt{-t_m}
        = \sqrt{-t_m}
    $$
    for $m\leq m_1$. 
    Thus, the right inequality is established.
    
    Next, we will show that $z(m) \geq 0$ for $m \leq m_1$.
    Observe that equation~(\ref{e:main_converta}) can be rearranged as
    $$
        z(m+1)=z(m)\left[ 1-\frac{\Delta t}{\epsilon}\big(z(m)+2\sqrt{-t_m}\big) \right]+\frac{\Delta t}{\sqrt{-t_m}+\sqrt{-t_{m+1}}}.
    $$
    When $\Delta t/\epsilon \leq 1/(3\sqrt{-t_0})$ and $m\leq m_1$,
    \begin{eqnarray*}
        1-\frac{\Delta t}{\epsilon}\big(z(m)+2\sqrt{-t_m}\big) 
            &\geq& 1-\frac{1}{3\sqrt{-t_0}}\big(\sqrt{-t_m}+2\sqrt{-t_m}\big)\\
            &\geq& 1-\frac{1}{3\sqrt{-t_0}}(3\sqrt{-t_0})\\
            &\geq& 0.
    \end{eqnarray*}
    Now, we are ready to prove the nonnegativeness of $z(m)$ for $m\leq m_1$ by induction.
    Indeed, when $m=0$, the assertion holds since $z(0)=\alpha\epsilon\geq 0$.
    Assume that when $m=k$, the assertion holds, that is, $z(k)\geq 0$.
    Now, for $m=k+1$,
    \begin{eqnarray*}
        z(k+1) &=& z(k)\left[ 1-\frac{\Delta t}{\epsilon}(z(k)+2\sqrt{-t_k}) \right]+\frac{\Delta t}{\sqrt{-t_k}+\sqrt{-t_{k+1}}}\\
            &\geq& \frac{\Delta t}{\sqrt{-t_k}+\sqrt{-t_{k+1}}}\\
            &\geq& 0,
    \end{eqnarray*}
   and so the assertion holds for $m=k+1$.
    Thus, by induction, $z(m)\geq 0$ for $m\leq m_1$.
\end{proof}

\medskip

\noindent
{\bf Lower bound of $z(m)$}.

Now, we can give a lower bound of $z(m)$. A combination of equation~(\ref{e:main_converta}) and Lemma \ref{l:main1} yields
\begin{align*}
    z(m+1) &= -\sqrt{-t_m}z(m)\big(\frac{\Delta t}{\epsilon\sqrt{-t_m}}z(m)+\frac{2\Delta t}{\epsilon}\big)+z(m)+\frac{\Delta t}{\sqrt{-t_m}+\sqrt{-t_{m+1}}}\\
    &\geq -\sqrt{-t_m}z(m)\big(\frac{\Delta t}{\epsilon}+\frac{2\Delta t}{\epsilon}\big)+z(m)+\frac{\Delta t}{2\sqrt{-t_{m}}}\\
    & \hspace{0.4cm} (\text{since $0\leq z(m)/\sqrt{-t_m}\leq 1$ by Lemma \ref{l:main1}})\\
    &= \big(-\frac{3\Delta t}{\epsilon}\sqrt{-t_m}+1\big)z(m)+\frac{\Delta t}{2\sqrt{-t_{m}}}.
\end{align*}
Then the solution of problem~(\ref{e:main_convert}) satisfies the following problem
\begin{subnumcases}
{\label{e:main_convert3}}
     z(m+1) \geq \big(-\frac{3\Delta t}{\epsilon}\sqrt{-t_m}+1\big)z(m)+\frac{\Delta t}{2\sqrt{-t_{m}}}, &\\
     z(0) = \alpha\epsilon. &
\end{subnumcases}
Use similar arguments for (\ref{e:main_convert2}),
we can apply the comparison lemma (Lemma \ref{l:compasrison}) to problems (\ref{e:main_convert}) and (\ref{e:main_convert3}) to obtain that
\begin{equation}\label{ineq:main1_proof2}
    z(m) \geq \prod_{j=0}^{m-1}(-\frac{3\Delta t}{\epsilon}\sqrt{-t_j}+1)z(0)+\sum_{j=0}^{m-1}\prod_{i=j+1}^{m-1}(-\frac{3\Delta t}{\epsilon}\sqrt{-t_i}+1)\frac{\Delta t}{2\sqrt{-t_j}}.
\end{equation}

Next, we can bound the right-hand side of (\ref{ineq:main1_proof2}) by an integration. 
Notice that $1-x\geq e^{-2x}$ for $0\leq x\leq \frac{1}{2}\ln{2}$. 
Hence, by noting the following constraint on $\Delta t/\epsilon$: 
$$
    \frac{3\Delta t}{\epsilon} \leq \frac{1}{2}\ln{2}
$$
(it can be done by the choice of $\Delta t$), inequality (\ref{ineq:main1_proof2}) becomes
\begin{eqnarray}
    z(m) 
        &\geq& \prod_{j=0}^{m-1}(e^{-6\Delta t\sqrt{-t_j}/\epsilon})z(0)+\sum_{j=0}^{m-1}\prod_{i=j+1}^{m-1}(e^{-6\Delta t\sqrt{-t_j}/\epsilon})\frac{\Delta t}{2\sqrt{-t_j}}\notag\\
        &=& e^{-6\Delta t(\sum_{j=0}^{m-1}\sqrt{-t_j})/\epsilon}z(0)+\sum_{j=0}^{m-1}e^{-6\Delta t(\sum_{i=j+1}^{m-1}\sqrt{-t_j})/\epsilon}\frac{\Delta t}{2\sqrt{-t_j}}.\label{ineq:main1_proof3}
\end{eqnarray}
Observe that $f(t)=\sqrt{-t}$ is monotone decreasing in $t$.
Hence, $\Delta t\sum_{i=j+1}^{m-1}\sqrt{-t_i}$ is a lower sum of $\int_{t_{j}}^{t_m} \sqrt{-t}~dt$,
and so $e^{-6\Delta t(\sum_{i=j+1}^{m-1}\sqrt{-t_j})/\epsilon}$ is bounded by
\begin{equation}\label{ineq:main1_proof4}
    e^{-6\Delta t(\sum_{i=j+1}^{m-1}\sqrt{-t_j})/\epsilon} \geq e^{-4((-t_{j})^{3/2}-(-t_m)^{3/2})/\epsilon} \quad \forall -1\leq j \leq m-2 . 
\end{equation}
Inequality (\ref{ineq:main1_proof3}) is thus transformed into
\begin{align}
    z(m) 
        &\geq e^{4(-t_m)^{3/2}/\epsilon}\big[ z(0)e^{-4(-t_0+\Delta t)^{3/2}/\epsilon}+\sum_{j=0}^{m-1}e^{-4(-t_j)^{3/2}/\epsilon}\frac{\Delta t}{2\sqrt{-t_j}}\big]\notag\\
        &\geq e^{4(-t_m)^{3/2}/\epsilon} \big[ z(0) e^{-4(-t_0+\Delta t)^{3/2}/\epsilon}+\int_{t_0}^{t_m}e^{-4(-t)^{3/2}/\epsilon}\frac{1}{2\sqrt{-t}}~dt \big]\notag\\
        &\geq\frac{1}{2} e^{4(-t_m)^{3/2}/\epsilon} \big[ \epsilon(2\alpha)e^{-4(-t_0)^{3/2}/\epsilon}+\int_{t_0}^{t_m}|t|^{-1/2}e^{-4(-t)^{3/2}/\epsilon}~dt \big].\label{ineq:main1_proof5}
\end{align}
The last inequality uses $z(0)=\alpha\epsilon$ and a similar argument as in inequality (\ref{ineq:main1_proof4}).
To proceed further, we need the following lemma by
Berglund and Gentz~\cite{Berglund:2006}.
\begin{lemma}\label{l:main2}
Fix constants $y_0<0$, and $\xi_0$, $c$, $p$, $q>0$. Then the function
$$
    \xi(y,\epsilon) = e^{c|y|^{p+1}/\epsilon}\left[ \epsilon\xi_0e^{-c|y_0|^{p+1}/\epsilon}+\int_{y_0}^y |u|^{q-1}e^{-c|u|^{p+1}/\epsilon}~du \right]
$$
satisfies
$$
    \xi(y,\epsilon)\asymp \epsilon|y|^{q-p-1} \quad 
    \forall\,y_0 \leq y \leq -\epsilon^{1/(p+1)}.
$$
\end{lemma}
Now, by applying Lemma \ref{l:main2} (with $p=q=1/2$, $y_0=-t_0$, $y=t_m$, $\xi_0=2\alpha$, $c=4$) to the right-hand side of the inequality (\ref{ineq:main1_proof5}), we can conclude that there exists a positive constant $C>0$ such that
\begin{equation}\label{ineq:main1_lowerbound}
    z(m) \geq C\frac{\epsilon}{|t_m|}
    \quad \forall m \leq m_1.
\end{equation}
This gives a lower bound of the solution $z(m)$ of problem (\ref{e:main_convert}).
Returning to the variable $(x(m),t_m)$, a combination of inequalities (\ref{ineq:main1_upperbound}) and (\ref{ineq:main1_lowerbound}) gives 
$$
    C\frac{\epsilon}{|t_m|} \leq x(m)-\sqrt{-t_m} \leq K\frac{\epsilon}{|t_m|} \quad \forall\; m\leq m_0,
$$
which proves (\ref{t:main1_temp30}).
\subsection{Corner layer: Proof of (\ref{t:main1_temp60})}

Now, we turn to the proof of 
(\ref{t:main1_temp60}).
For this, we study the evolution of the solution of problem~(\ref{e:main}) for 
time larger than $t_{m_1}$.
To proceed, 
note that 
$$
t_{m_1} = -c_1^*\epsilon^{2/3} 
\;\, \text{for some $c_1^*\geq c_1$ with $c_1^*=O(1)$}.
$$
Using this, and the inequalities (\ref{ineq:main1_upperbound}) and (\ref{ineq:main1_lowerbound}),
it follows that 
the value $x(m_1)$ of the solution $x$
of problem~(\ref{e:main}) is given by 
$$
    x(m_1) = K^*\epsilon^{1/3}
$$
for some positive constant $K^*$ satisfying
\begin{equation}\label{ineq:main1_K*}
    \big( \frac{C}{c_1^*}+\sqrt{c_1^*} \big) \leq K^* \leq \big( \frac{K}{c_1^*}+\sqrt{c_1^*} \big).
\end{equation}
Therefore, we consider the following initial value problem
\begin{subnumcases}
{\label{e:main2}}
    x(m+1)=-\frac{\Delta t}{\epsilon}x(m)^2+x(m)-\frac{\Delta t}{\epsilon}t_m, &{\label{e:main2a}}\\
    x(m_1)=K^*\epsilon^{1/3}. &{\label{e:main2b}}
\end{subnumcases}
Now, make the change of variables
$$
    x(m)=\epsilon^{1/3} y(m),\;\, t_m=\epsilon^{2/3} s_{m},\;\, \Delta t=\epsilon^{2/3} \Delta s
$$
Then, problem (\ref{e:main2}) is converted into
\begin{subnumcases}
{\label{e:main2_convert}}
    y(m+1)=-(\Delta s) y(m)^2+y(m)-(\Delta s) s_m, \; m \geq m_1 &{\label{e:main2_converta}}\\
    y(m_1)=K^*. &{\label{e:main2_convertb}}
\end{subnumcases}
Note that
$$
    s_{m} = -c_1^*+(m-m_1)\Delta s \quad \forall m \geq m_1.
$$
Recall that $t_m \leq 0$ for $m\leq m_0$, and hence that 
$$
  s_m\leq 0 \;\text{ for }\; m\leq m_0.
$$  

\noindent
{\bf Decreasing property of $\{y(m)\}_{m=m_1}^{m_0}$.}

We will show that $y(m) > 0$ for $m\leq m_0$.
To do this, we will first show that $\{y(m)\}$ is decreasing for $m_1 \leq m\leq m_0$ whenever 
$$ 
    \Delta s < 1/[2K/(c_1^*+\sqrt{c_1^*})],
$$ 
which is equivalent to 
\begin{equation}\label{e:main_ds}
\Delta t/\epsilon < \epsilon^{-1/3}/[2K/(c_1^*+\sqrt{c_1^*})].
\end{equation}
We remark that the $\epsilon_0$ in the statement of Theorem~\ref{t:main1} is chosen so that (\ref{e:main_ds}) holds.
This assertion can be proved by induction on $m$.
To proceed, we first note that 
\begin{equation}\label{ineq:main1_proof6}
    1-2(\Delta s)y(m_1)=1-2(\Delta s)K^*>1-2\frac{1}{2K/(c_1^*+\sqrt{c_1^*})}K^* \geq 0.
\end{equation}
Now, when $m=m_1$, by (\ref{e:main2_converta}) and 
the leftmost part of inequality (\ref{ineq:main1_K*}), 
$$
    y(m_1+1)-y(m_1)=-\Delta s(y(m_1)^2+s_{m_1})=-\Delta s((K^*)^2-c_1^*) < 0,
$$
and so the assertion holds for $m=m_1$.
Assume that the assertion holds for $m=k$, 
that is $y(k+1)-y(k)=-\Delta s(y(k)^2+s_k)\leq 0$.
When $m=k+1$,
\begin{eqnarray*}
    y(k+2)-y(k+1)
        &=& -\Delta s(y(k+1)^2+s_{k+1})\\
        &=& -\Delta s\big[ \big(y(k)-\Delta s(y(k)^2+s_k)\big)^2+s_{k+1} \big]\\
        &=& -\Delta s\big[ y(k)^2-2\Delta s(y(k)^2+s_k)y(k)+\Delta s^2(y(k)^2+s_k)^2+s_{k+1} \big]\\
        &\leq& -\Delta s \big[ y(k)^2-2\Delta s(y(k)^2+s_k)y(k)+s_k \big]\\
        && (\text{since $s_{k+1}\geq s_k$})\\
        &=& -\Delta s \big[ \big(y(k)^2+s_k\big)\big(1-2(\Delta s) y(k)\big) \big]\\
        &<& 0, \hspace{0.4cm} \big(\text{since $y(k)^2+s_k \geq 0$, $y(k)\leq y(m_1)$ and inequality (\ref{ineq:main1_proof6})\big)}
\end{eqnarray*}
and so the assertion holds for $m=k+1$.
Thus, by induction, $\{y(m)\}$ is decreasing for $m_1 \leq m\leq m_0$.

\medskip
\noindent
{\bf Positivity of $\{y(m)\}_{m=m_1}^{m_0}$}.

With the aid of the decreasing property of $\{y(m)\}_{m=m_1}^{m_0}$,
we are ready to prove that
\begin{equation}\label{e:main2_y2}
    y(m)>0 \quad \text{for}\; m_1\leq m\leq m_0.
\end{equation}
This can be shown by induction on $m$.
Indeed, when $m=m_1$, $y(m_1)=K^*>0$, and so the assertion holds for $m=m_1$.
Assume that the assertion holds for $m=k$, that is, $y(k)>0$.
Now, for $m=k+1$,
\begin{eqnarray*}
    y(k+1) 
    &=& -\Delta s y(k)^2+y(k)-(\Delta s)s_k\\
    &=& \big( 1-(\Delta s)y(k) \big)y(k)-(\Delta s)s_k\\
    &\geq& \big( 1-(2\Delta s)y(m_1) \big)y(k)-(\Delta s)s_k\\
    &&\big( \text{since $\{ y(m) \}_{m=m_1}^{m_0}$ is decreasing and $y(k)$ is positive} \big)\\
    &>& 0, \hspace{0.4cm} \big( \text{using the fact that $y(k)>0$ and $s_k < 0$, and inequality (\ref{ineq:main1_proof6})} \big)
\end{eqnarray*}
and so the assertion holds for $m=k+1$.
Hence, by induction, $y(m)>0$ for $m_1\leq m\leq m_0$.

\medskip
\noindent
{\bf Decreasing property of $\{y(m)\}_{m=m_0}^{\infty}$.}

Next, we will show that
$y(m)$ is decreasing for $m > m_0$, and becomes negative for $m > m_2$ and for some $m_2 > m_0$.
Indeed, since $s_m > 0$ for $m > m_0$, we have 
\begin{eqnarray*}
    y(m+1)-y(m) 
    &=& -(\Delta s) y(m)^2-(\Delta s) s_m\\
    &<& -(\Delta s) s_{m_0+1} \quad \forall m > m_0.
\end{eqnarray*}
Together with the fact that $y(m_0)>0$ and $(\Delta s) s_{m_0+1}$ is a positive constant,
we can conclude that 
there exists an integer $m_2>m_0$ such that $y(m)>0$ for $m_0<m<m_2$, $y(m_2) \geq 0$ 
and $y(m)<0$ for $m > m_2$. 
Set 
\begin{equation} \label{e:c2_temp20}
c_2 = s_{m_2}.
\end{equation}
In Lemma~\ref{e:c_2}, we will establish that the magnitude of $c_2$ is of order $O(1)$.
Then with the use of this fact and returning to the origin variables $(t_m,x(m))$, we conclude that $x(m)\asymp \epsilon^{1/3}$ for $m_0 \le m \le m_2$,
and so for $-c_1\epsilon^{2/3}\leq t_m \leq c_2\epsilon^{2/3}$.
This completes the proof of (\ref{t:main1_temp60}).

\medskip

\begin{lemma} \label{e:c_2}
    $c_2$ is a positive constant of order $O(1)$.
\end{lemma}
\begin{proof}
First, we give an upper bound for $m_2$. 
Now, for $m>m_0$, $y(m)$ can be estimated as follows: 
\begin{eqnarray*}
    y(m)
    &=& y(m_0+1)+\sum_{i=m_0+1}^{m-1}\big( y(i+1)-y(i) \big)\\
    &<& y(m_0+1)+\sum_{i=m_0+1}^{m-1}\big( -(\Delta s)s_i \big)\\
    &=& y(m_0+1)-\Delta s\sum_{i=m_0+1}^{m-1}\big(  s_{m_0+1}+(i-m_0-1)\Delta s \big)\\
    &<& y(m_0+1)-\frac{(m-m_0-2)(m-m_0-1)}{2}(\Delta s)^2.
\end{eqnarray*}
Define $m_2^+$ by 
\begin{equation*}
    m_2^+:=\left\lceil \frac{\sqrt{2y(m_0+1)}}{\Delta s}+m_0+2 \right\rceil.
\end{equation*}
Therefore, if $m \geq m_2^+$, 
we have 
\begin{eqnarray*}
    y(m)
    &<& y(m_0+1)-\left(\frac{\sqrt{2y(m_0+1)}}{\Delta s}\right)\left(\frac{\sqrt{2y(m_0+1)}}{\Delta s}+1\right)\frac{(\Delta s)^2}{2}\\
    &<& y(m_0+1)-\left(\frac{\sqrt{2y(m_0+1)}}{\Delta s}\right)^2\frac{(\Delta s)^2}{2}\\
    &=& 0.
\end{eqnarray*}
Note that $s_{m_2^+}$ satisfies
\begin{eqnarray*}
    s_{m_2^+}
    &=& -c_1^*+(m_2^+-m_1)\Delta s\\
    &<& -c_1^*+\left( \frac{\sqrt{2y(m_0+1)}}{\Delta s}+m_0+3-m_1 \right)\Delta s\\
    &=& -c_1^*+(m_0+3-m_1)\Delta s+\sqrt{2y(m_0+1)}\\
    &=& s_{m_0+3}+\sqrt{2y(m_0+1)}.
\end{eqnarray*}
Note that $y(m_0+1)$ is a positive constant of order $O(1)$, and that $\Delta s$ is at most of order $O(\epsilon^{1/3})$ and $0 < s_{m_0+3}=s_{m_0}+3\Delta s<3\Delta s$ since $s_{m_0} < 0$. Taken together, it follows from the above inequality that $s_{m_2^+}$ is at most of order $O(1)$. Since $y(m)$ is decreasing, $y(m_2^+)<0$, and $y(m_2)\geq 0$, we thus have  $m_2<m_2^+$.

Next, we give a lower bound for $m_2$. For $m_0<m<m_2^+$, $y(m)$ can be estimated as follows:
\begin{eqnarray*}
    y(m)
    &=& y(m_0+1)+\sum_{i=m_0+1}^{m-1}\big( y(i+1)-y(i) \big)\\
    &=& y(m_0+1)-\Delta s\sum_{i=m_0+1}^{m-1}\big( y(i)^2+s_i \big)\\
    &>& y(m_0+1)-\Delta s\sum_{i=m_0+1}^{m-1}\big( y(m_0+1)^2+s_{m_2^+} \big)\\
    & &(\text{since $\{ y(i) \}$ is decreasing and $s_i<s_{m_2^+}$})\\
    &=& y(m_0+1)-(y(m_0+1)^2+s_{m_2^+})(m-m_0-1)\Delta s.
\end{eqnarray*}
Define $m_2^-$ by
\begin{equation*}
    m_2^-:= \left\lfloor \frac{y(m_0+1)}{(y(m_0+1)^2+s_{m_2^+})\Delta s}+m_0+1 \right\rfloor.
\end{equation*}
It then follows from the above inequality that
\begin{eqnarray*}
    y(m)
    &>& y(m_0+1)-(y(m_0+1)^2+s_{m_2^+})\frac{y(m_0+1)}{(y(m_0+1)^2+s_{m_2^+})\Delta s}\Delta s\\
    &=& 0 \quad \forall m\leq m_2^-.
\end{eqnarray*}
In particular, $y(m_2^-) > 0$.
Now, $s_{m_2^-}$ satisfies
\begin{eqnarray*}
    s_{m_2^-}
    &=& -c_1^*+(m_2^--m_1)\Delta s\\
    &>& -c_1^*+\left( \frac{y(m_0+1)}{(y(m_0+1)^2+s_{m_2^+})\Delta s}+m_0+1-m_1 \right)\Delta s\\
    &=& -c_1^*+(m_0+1-m_1)\Delta s+\frac{y(m_0+1)}{y(m_0+1)^2+s_{m_2^+}}\\
    &=& s_{m_0+1}+\frac{y(m_0+1)}{y(m_0+1)^2+s_{m_2^+}}.
\end{eqnarray*}
Note that $y(m_0+1)$ is a positive constant of order $O(1)$, and that $s_{m_2^+}$ is a positive constant of at most order $O(1)$.
Taken together, it follows from the above inequality that $s_{m_2^-}>0$ is at least of order $O(1)$. Since $y(m)$ is decreasing, $y(m_2^-)>0$, and $m_2$ is the largest integer such that $y(m)\geq 0$, 
we can thus deduce $m_2\geq m_2^-$.

Taken together, we have that $s_{m_2^-} \leq s_{m_2} \leq s_{m_2^+}$.
Since $s_{m_2^-}$ is a constant of at least order $O(1)$ and $s_{m_2^+}$ is a constant of at most order $O(1)$, we can conclude that $c_2=s_{m_2}$ is a constant of order $O(1)$. The proof is thus completed.
\end{proof}

\subsection{Beyond corner layer: Behavior of $y(m)$ for $m$ such that $t_m\geq c_2\epsilon^{2/3}$}

To proceed, as in the last section, we consider the variables $(y(m),s_m)$. 
Equation (\ref{e:main2_converta}) gives that $y(m+1)-y(m)\leq -(\Delta s)s_m$. Hence, we have the following estimate.
\begin{eqnarray*}
    y(m_2+1+n)&\leq& y(m_2+1)-\Delta s \sum_{i=m_2+1}^{m_2+1+n} s_i\\
        &=& y(m_2+1)-\Delta s \sum_{i=0}^{n}(s_{m_2+1}+i\Delta s)\\
        &=& y(m_2+1)-(n+1)(\Delta s )s_{m_2+1}-\frac{n(n+1)}{2}(\Delta s)^2
\end{eqnarray*}
This shows that $\{ x(m) \}_{m=m_2+1}^\infty$ is decreasing and approaches negative infinity
as $m\to\infty$.

\subsection{Dependence of tipping time on the sweep rate $\epsilon$}

Theorem \ref{t:main1} depicts the nature of solutions to problem~(\ref{e:main}). Due to the critical transition of the system, we are particularly interested in predicting the tipping time $t_{\mt}$. Theorem \ref{t:main1} showed that $t_{\mt}=C\epsilon^{2/3}$ for some positive constant $C$. We would like to further understand how $C$ will change in response to variations in $\epsilon$ and $\Delta t$. Here, we consider problem~(\ref{e:main}) with initial conditions $t_0=-1$ and $x(0)=\sqrt{-t_0}+\epsilon$, i.e., 
\begin{equation}\label{e:tipping}
    \begin{cases}
        x(m+1)=-\frac{\Delta t}{\epsilon}x(m)^2+x(m)-\frac{\Delta t}{\epsilon}t_m, &\\
        x(0)= \sqrt{-t_0}+\epsilon,  & \\
        t_0 = -1, &
    \end{cases}
\end{equation}
where $\epsilon$ and $\Delta t$ will be specified later.
Note that the tipping time $t_{\mt}$ is negative for this case due to Theorem \ref{t:main1}.
Now, we fix the ratio $\Delta t/\epsilon$ to be $\ln(2)/6$ and use the solution of problem~(\ref{e:tipping}) to compute the tipping time $t_{\mt}$ for various $\epsilon$.
Then the obtained pairs $(\log(\epsilon), \log(t_{\mt}))$ are depicted as markers in
 Figure \ref{fig:time_order}.
As indicated in Figure \ref{fig:time_order}, $\log(t_{\mt})$ seems to be a linear function of $\log(\epsilon)$ with slope $0.6667$ and $y$-axis intercept $C_0 \approx 0.0232$.
We can thus conclude that $t_{\mt} = e^{C_0} \epsilon^{2/3}$ for small $\epsilon$, and so the crucial constant $C$ given in Theorem \ref{t:main1} is conjectured to be equal to $e^{C_0}$.

Next, we wonder whether this property still holds when $\Delta t/\epsilon$ takes on different values and even when $\Delta t=o(\epsilon)$. 
Some numerical attempts with $\Delta t = O(\epsilon^\alpha)$ and $\alpha \ge 1$ are summarized in Table \ref{table:tippingtime_const}.
It turns out that as long as $\Delta t/\epsilon$ satisfies the condition $\Delta t/\epsilon\in (0,\delta_0)$, $C$ tends to a constant as $\epsilon \to 0^+$. 


\begin{figure}[ht]
    \centering
    \begin{subfigure}[b]{0.49\textwidth}
        \centering
        \includegraphics[width=\textwidth]{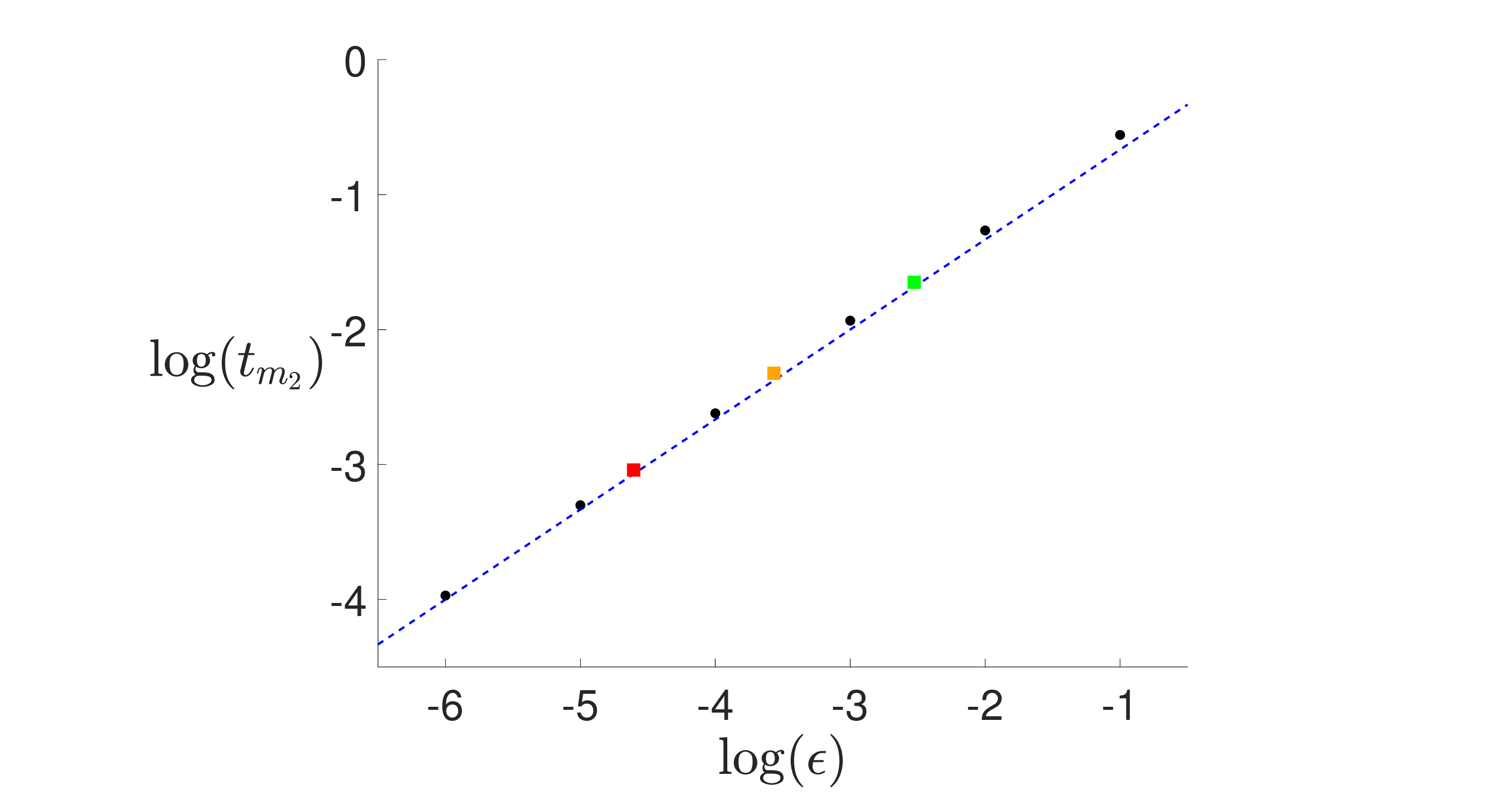}
        \caption{\label{fig:time_order}}
    \end{subfigure}
    \hfill
    \begin{subfigure}[b]{0.49\textwidth}
        \centering
        \includegraphics[width=\textwidth]{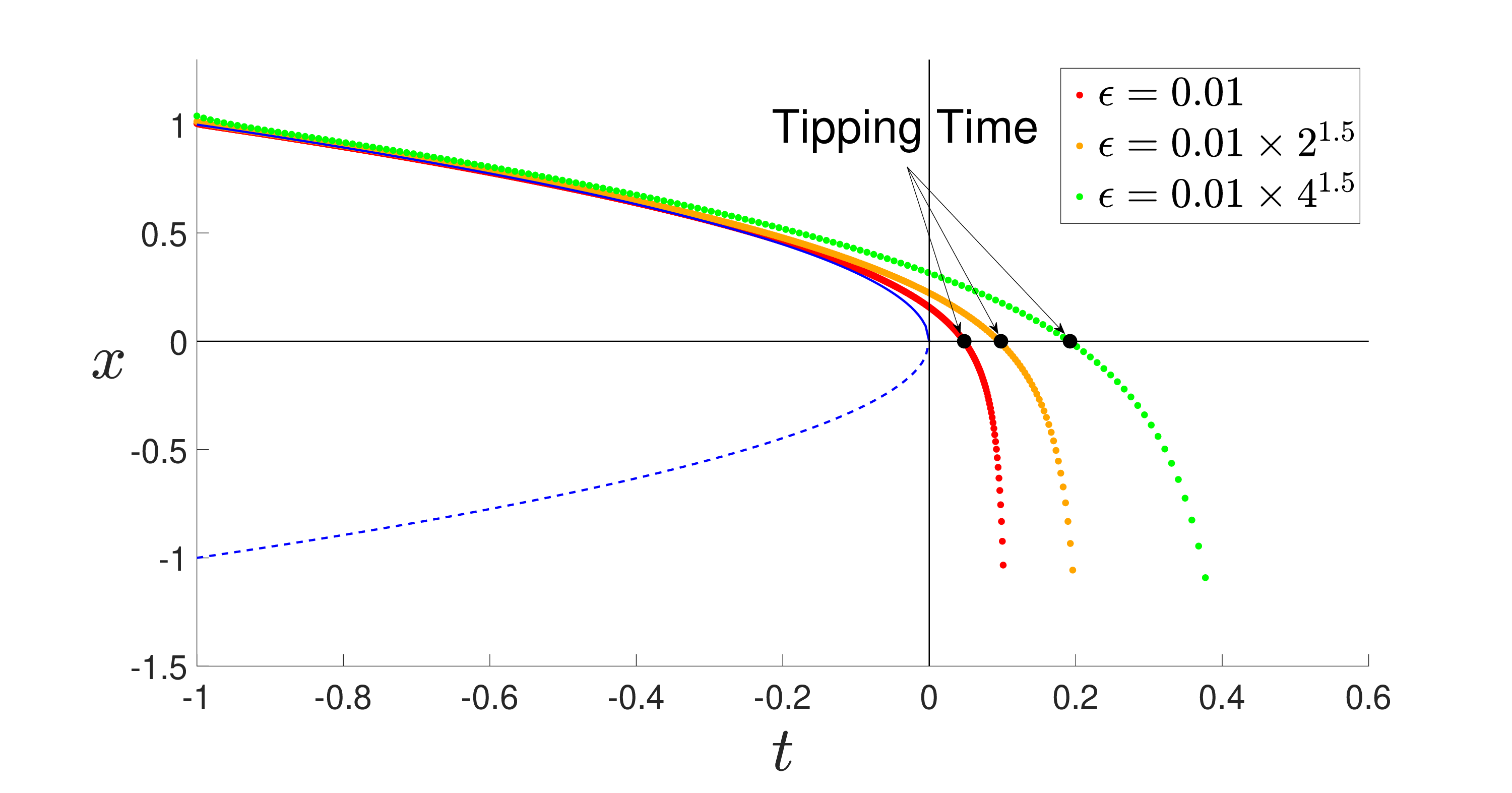}
        \caption{\label{fig:multi_sol}}
    \end{subfigure}
    \caption{
    (a). The dependence of tipping time $t_{\mt}$ on the sweep rate parameter $\epsilon$ ($\epsilon=0.01$ (red marker), $\epsilon=0.01\times 2^{1.5}$ (orange marker), $\epsilon=\epsilon=0.01\times 4^{1.5}$ (green marker)). The blue dashed line is $\log(t)= (2/3) \log(\epsilon)$. 
    (b). Solution profiles corresponding to the red, the orange, and the green markers in panel (a), respectively. 
    Here $\Delta t/\epsilon$ is fixed to be $\ln(2)/6$.
    }
\end{figure}


 \begin{table}[ht!]
    \centering
    \begin{tabular}{ |l|c c c c c c| } 
         \hline
        \quad $t_{\mt}/\epsilon^{2/3}$ & $\epsilon=1$ &  $\epsilon=10^{-1}$ & $\epsilon=10^{-2}$ & $\epsilon=10^{-3}$ & $\epsilon=10^{-4}$ & $\epsilon=10^{-5}$\\ 
        \hline
        $\Delta t=C\epsilon$ & 
        1.0794 & 1.0423 & 1.0300 & 1.0262 & 1.0204 & 1.0204\\
        $\Delta t=(3C/4)\epsilon$ & 
        1.0794 & 1.0289 & 1.0238 & 1.0262 & 1.0231 & 1.0210\\
        $\Delta t=(C/2)\epsilon$ & 
        1.0794 & 1.0423 & 1.0300 & 1.0204 & 1.0204 & 1.0191\\
        $\Delta t=(C/4)\epsilon$ & 
        1.0794 & 1.0289 & 1.0238 & 1.0204 & 1.0204 & 1.0191\\
        $\Delta t=C\epsilon^{1.5}$ & 
        1.0794 & 1.0389 & 1.0200 & 1.0192 & 1.0189 & 1.0188\\
        $\Delta t=C\epsilon^2$ & 
        1.0794 & 1.0209 & 1.0190 & 1.0188 & 1.0188 & 1.0188\\
        $\Delta t=C\epsilon^{2.5}$ & 
        1.0794 & 1.0202 & 1.0188 & 1.0188 & 1.0188 & 1.0188\\
        \hline
        \end{tabular}
    \caption{
    The ratio $t_{\mt}/\epsilon^{2/3}$ for various $\Delta t$ and $\epsilon$.
 The numerical result indicates that $c_2=t_{\mt}/\epsilon^{2/3}$ approaches to the same constant ($\approx 1.0188$) as $\epsilon$ tends to 0. 
 Here the constant $C$ is chosen to be $\ln(2)/6\approx 0.1155$ to satisfy the constraint given in Theorem \ref{t:main1}.
 }
    \label{table:tippingtime_const}
\end{table}



\section{Solution behavior when $\epsilon/\Delta t = o(1)$} \label{sec:negative tipping time}
\setcounter{equation}{0}

In the last section, we have provided a detailed description of the properties of $x(m)$ under the condition $\epsilon/\Delta t=O(1)$, or more precisely, when $\Delta t/\epsilon\in(0,\delta_0)$. 
In this section, we proceed to consider the scenario when $\epsilon/\Delta t = o(1)$. In this scenario, $x(m)$ may be negative. However, in the context of most realistic systems, the solution $\{x(m)\}_{m\in\bN}$ is non-negative. Therefore, we restrict our consideration to such a case throughout this section.

\subsection{Negative tipping time}

Now we would like to know under which conditions $x(1)$ becomes the tipping point. 
This question can be addressed by the following lemma which is the assertion (a) of Theorem~\ref{t:main2}.
\begin{lemma}[\bf Tipping at $m=1$]
    Let $\{ x(m) \}_{m\in\mathbb{N}}$ be the solution of problem~(\ref{e:main}).
    Suppose $\Delta t>1/\alpha$. 
     Then $x(m)<-\sqrt{-t_m}$ at $m=1$,
     that is, $x(1)$ is the tipping point.
\end{lemma}
\begin{proof}
    To begin with, from equation~(\ref{e:maina_convert}) and initial condition $x(0)=\sqrt{-t_0}+\alpha\epsilon$, $x(1)$ can be estimated as follows:
    \begin{eqnarray*}
        x(1)
        &=& x(0)-\frac{\Delta t}{\epsilon}(x(0)+\sqrt{-t_0})(x(0)-\sqrt{-t_0})\\
        &=& \sqrt{-t_0}+\alpha\epsilon-\frac{\Delta t}{\epsilon}(2\sqrt{-t_0}+\alpha\epsilon)(\alpha\epsilon)\\
        &<& \sqrt{-t_0}+\alpha\epsilon-(2\sqrt{-t_0}+\alpha\epsilon)
        \hspace{0.4cm} \big(\text{since $\Delta t>1/\alpha$}\big)\\
        &=& -\sqrt{-t_0}
    \end{eqnarray*}
    Together with the fact that $-\sqrt{-t_0}<-\sqrt{-t_1}$, $x(1)$ is less than $-\sqrt{-t_1}$. 
    Now, by Lemma~\ref{l:earlytipping}, we can conclude that $t_1$ is the tipping time when $\Delta t$ exceeds $1/\alpha$.
    The proof is thus completed.
\end{proof}

\medskip
Next, if $x(1)$ is not the tipping point, then this raises a question: when is the next potential tipping time? 
Suppose that $-\sqrt{t_1}<x(1)<\sqrt{t_1}$. 
Then equation~(\ref{e:maina_convert}) indicates that $x(2) > x(1)$,
and thus that $x(2)$ cannot be a tipping point and so
$x(3)$ could be the next potential tipping point. 
On the other hand, the following lemma provides a sufficient condition for $t_3$ to be the tipping time,
which establishes the assertion (b) of Theorem~\ref{t:main2}.
\begin{lemma}[\bf Tipping at $m=3$]\label{l:tippig_at_t3}
    Let $\{ x(m) \}_{m\in\mathbb{N}}$ be the solution of problem~(\ref{e:main}) and set $\Delta t=C\epsilon^b$.
    Suppose that $b\in(0,1/2)$ and $\alpha>1/(-4t_0)$.
    Then there exists a small $\epsilon_0>0$ such that for $\epsilon\in(0,\epsilon_0)$, the solution $\{x(m)\}$ satisfies 
    \begin{enumerate}[\rm(i)]
        \item 
        $x(m) > 0$ and $(-1)^m(x(m) - \sqrt{-t_m})>0$ for $m=1, 2$; and
        \item 
        $x(m)<-\sqrt{-t_m}$ at $m=3$.
    \end{enumerate}
\end{lemma}
\begin{proof}
    The proof consists of three steps.

    {\it Step 1: There exists a $\epsilon_2\in(0,1)$ such that $x(1) > 0$ and $x(1) - \sqrt{-t_1}<0$ for all $\epsilon \in (0, \epsilon_2)$}.\\
    From equation~(\ref{e:maina_convert}) and initial condition $x(0)=\sqrt{-t_0}+\alpha\epsilon$ , $x(1)>0$ is equivalent to
    \begin{eqnarray*}
        \frac{\Delta t}{\epsilon}
        &<& \frac{x(0)}{x(0)+\sqrt{-t_0}}\cdot \frac{1}{x(0)-\sqrt{-t_0}}\\
        &=&  \big( \frac{1}{\alpha\epsilon} \big)\frac{\sqrt{-t_0}}{2\sqrt{-t_0}+\alpha\epsilon}.
    \end{eqnarray*}
    Since $\Delta t=C\epsilon^b$ with $b\in(0,1/2)$, there exists a constant $\epsilon_1 \in (0, 1)$ such that the above inequality holds for $\epsilon\in(0,\epsilon_1)$.
    Now, using equation~(\ref{e:maina_convert}) and initial condition $x(0)=\sqrt{-t_0}+\alpha\epsilon$, and applying Taylor's expansion to $\sqrt{-t_1}$ at $-t_0$, $\sqrt{-t_1}-x(1)$ can be estimated as follows:
    \begin{equation}\label{e:3StepTipping1}
        \begin{split}
            \sqrt{-t_1}-x(1)
            &=\sqrt{-t_1}-\big( \sqrt{-t_0}+\alpha\epsilon-\frac{\Delta t}{\epsilon}(2\sqrt{-t_0}+\alpha\epsilon)\alpha\epsilon \big)
            \quad \mbox{(by Eq.~(\ref{e:maina_convert}))}\\
            &= \big( \sqrt{-t_0}-\frac{\Delta t}{2\sqrt{-t_0}}+O(\Delta t^2) \big)-\big( \sqrt{-t_0}+\alpha\epsilon-\frac{\Delta t}{\epsilon}(2\sqrt{-t_0}+\alpha\epsilon)\alpha\epsilon \big)\\
            &= \big( 2\alpha\sqrt{-t_0}-\frac{1}{2\sqrt{-t_0}} \big)\Delta t-\alpha\epsilon+\alpha^2\Delta t\epsilon+O(\Delta t^2)\\
            &= \big( 2\alpha\sqrt{-t_0}-\frac{1}{2\sqrt{-t_0}} \big)C\epsilon^b- \alpha\epsilon+\alpha^2C\epsilon^{1+b}+O(\epsilon^{2b}).
        \end{split}
    \end{equation}
    Since $\alpha>1/(-4t_0)$, the coefficient of the term $\epsilon^b$ is positive. 
    Together with $0<b<1/2$, the right hand side of equation~(\ref{e:3StepTipping1}) is positive 
    for $\epsilon\in(0,\epsilon_2)$ and for some $\epsilon_2 \in (0, \epsilon_1)$,
    and thus $\sqrt{-t_1}-x(1)>0$ for $\epsilon\in(0,\epsilon_2)$.
    
    {\it Step 2: There exists a $\epsilon_0\in(0,\epsilon_2)$ such that $x(2) > 0$ and $x(2) - \sqrt{-t_2}>0$ for all $\epsilon \in (0, \epsilon_0)$}.\\
     To begin with, we claim that there exists a small $\epsilon_0 \in (0,1)$
    \begin{equation}\label{ineq:3StepTipping1}
        \frac{\Delta t}{\epsilon}>\frac{K}{\sqrt{-t_1}-x(1)} \quad \forall \epsilon \in (0, \epsilon_0)
    \end{equation}
    where $K:=(1+\sqrt{5})/2$. 
   Indeed, with the use of $\Delta t=C\epsilon^b$ and equation~(\ref{e:3StepTipping1}),  
   inequality~(\ref{ineq:3StepTipping1}) is equivalent to
    \begin{equation*}
        C^{-1}\epsilon^{1-b}<\frac{1}{K}\big\{ \big( 2\alpha\sqrt{-t_0}-\frac{1}{2\sqrt{-t_0}} \big)C\epsilon^b-\alpha\epsilon+\alpha^2C\epsilon^{1+b}+O(\epsilon^{2b}) \big\},
    \end{equation*}
    which, after a rearrangement, 
   is equivalent to the following equation
    \begin{equation}\label{ineq:3StepTipping2}
        \frac{C^2}{K}\big( 2\alpha\sqrt{-t_0}-\frac{1}{2\sqrt{-t_0}} \big)\epsilon^{2b-1}-\alpha C \epsilon^b+\alpha^2C^2\epsilon^{2b}+O(\epsilon^{3b-1})>1.
    \end{equation}
    Recall that $\alpha>1/(-4t_0)$ and $2b-1<0$. 
    Then inequality~(\ref{ineq:3StepTipping2}) holds $\epsilon\in(0,\epsilon_0)$ and for some $\epsilon_0 \in (0, \epsilon_2)$, and thus the assertion of the claim is established.

  Now, we are ready to establish the assertion of this step.
  Indeed, since $x(1)>0$ and $\sqrt{-t_2} < \sqrt{-t_1}$, we have    
    $\big( \sqrt{-t_2}-x(1) \big)/\big( \sqrt{-t_1}+x(1) \big) < 1 < K$.
  It then follows from (\ref{ineq:3StepTipping1}) that
   \begin{equation*}
        \frac{\Delta t}{\epsilon} 
        >  \frac{K}{\sqrt{-t_1}-x(1)}
        > \frac{\sqrt{-t_2}-x(1)}{\sqrt{-t_1}+x(1)}\cdot\frac{1}{\sqrt{-t_1}-x(1)} \quad \forall \epsilon \in (0, \epsilon_0),
    \end{equation*}
    and so
     \begin{equation*}
        \frac{\Delta t}{\epsilon} 
        > \frac{\sqrt{-t_2}-x(1)}{\sqrt{-t_1}+x(1)}\cdot\frac{1}{\sqrt{-t_1}-x(1)} \quad \forall \epsilon \in (0, \epsilon_0).
    \end{equation*}  
    Note that $\sqrt{-t_1}-x(1) > 0$ for $\epsilon \in (0, \epsilon_0)$ due to Step 1.
    Then for each $\epsilon \in (0, \epsilon_0)$,
    the inequality $x(2)>\sqrt{-t_2}$ is a consequence of a rearrangement of the above inequality and equation~(\ref{e:maina_convert}).

    {\it Step 3: The inequality $x(3) < -\sqrt{-t_3}$ holds for all $\epsilon\in(0,\epsilon_0)$}.\\
 To proceed, 
 let $X^+$ be the unique positive root of the quadratic polynomial
 $$
    p(X) = (-t_1-x(1)^2)X^2-(\sqrt{-t_1}-x(1))X-1.
 $$
 We claim that $p(\Delta t/\epsilon) > 0$ for $\epsilon\in(0,\epsilon_0)$.
  %
    To see this, 
    from the exact expression of $X^+$ we have 
    \begin{eqnarray*}
        X^+
        &=& \frac{1}{2(-t_1-x(1)^2)}\big( (\sqrt{-t_1}-x(1))+\sqrt{(\sqrt{-t_1}-x(1))^2+4(-t_1-x(1)^2)} \big)\\
        &<& \frac{1}{2(-t_1-x(1)^2)}\big( (1+\sqrt{5})(\sqrt{-t_1}+x(1)) \big) \hspace{0.4cm} \big(\text{since $0<x(1)$}\big)\\
        &=& \frac{1+\sqrt{5}}{2} \cdot \frac{1}{\sqrt{-t_1}-x(1)}\\
        &=& \frac{K}{\sqrt{-t_1}-x(1)}.
    \end{eqnarray*}
   Together with inequality~(\ref{ineq:3StepTipping1}),
   we have that $\Delta t/\epsilon > X^+$ for $\epsilon\in(0,\epsilon_0)$.
   Note that the coefficient of the quadratic term of $p(X)$ is 
   positive due to the fact that ${\sqrt{-t_1}-x(1)}>0$ for $\epsilon\in(0,\epsilon_0)$. 
   Thus $p(X)$ is increasing for $X > X^+$.
   Taken together, we can thus deduce that 
   $p(\Delta t/\epsilon) > p(X^+) = 0$ for $\epsilon\in(0,\epsilon_0)$.
   
    
    Next, a rearrangement of the inequality $p(\Delta t/\epsilon) > 0$
    gives 
    \begin{equation*}
        \frac{\epsilon}{\Delta t}<-(\sqrt{-t_1}-x(1))+\frac{\Delta t}{\epsilon}(\sqrt{-t_1}+x(1))(\sqrt{-t_1}-x(1)),
    \end{equation*}
    which, after taking the reciprocal of both sides, gives 
    \begin{eqnarray*}
        \frac{\Delta t}{\epsilon}
        &>& \frac{1}{(\sqrt{-t_1}-\sqrt{-t_2})-(\sqrt{-t_1}-x(1))+\frac{\Delta t}{\epsilon}(\sqrt{-t_1}+x(1))(\sqrt{-t_1}-x(1))}\\
        & &\text{(since $\sqrt{-t_1}>\sqrt{-t_2}$)}\\
        &=& \frac{1}{x(1)+\frac{\Delta t}{\epsilon}(\sqrt{-t_1}+x(1))(\sqrt{-t_1}-x(1))-\sqrt{-t_2}}\\
        &=& \frac{1}{x(2)-\sqrt{-t_2}} \hspace{0.4 cm} \text{(by equation~(\ref{e:maina_convert}))}\\
        &>& \frac{x(2)+\sqrt{-t_3}}{x(2)+\sqrt{-t_2}}\cdot\frac{1}{x(2)-\sqrt{-t_2}}. \hspace{0.4cm} \text{(since $\sqrt{-t_2}>\sqrt{-t_3}$ and $x(2)-\sqrt{-t_2}>0$)}
    \end{eqnarray*}
    In conclusion, we have
    $$
     \frac{\Delta t}{\epsilon} > \frac{x(2)+\sqrt{-t_3}}{x(2)+\sqrt{-t_2}}\cdot\frac{1}{x(2)-\sqrt{-t_2}} \hspace{0.4cm} \quad \forall \epsilon \in (0, \epsilon_0). 
    $$
    Then for each $\epsilon \in (0, \epsilon_0)$,
    the inequality $x(3)<-\sqrt{-t_3}$ is a consequence of a rearrangement of the above inequality and equation~(\ref{e:maina_convert}).
    The proof is thus completed.
\end{proof}




\subsection{Classification of solutions} 

Based on Theorem \ref{t:main1} and \ref{t:main2}, 
the solution $\{x(m)\}_{m\in\bN}$ of problem~(\ref{e:main}) can be classified into three types based on their dynamical behavior:

\begin{enumerate}[(I).]
    \item $x(m)$ stays above the stable manifold and has a bifurcation delay. In other words, $x(m)>\sqrt{-t_m}$ for $m < \mt$ and the tipping time $t_{\mt}$ is positive.
    \item 
    $x(m)$ oscillates around the stable manifold before the tipping time and has a bifurcation delay. 
    Precisely, $x(m) > 0$ and $(-1)^m(x(m) - \sqrt{-t_m})>0$ for $m < \mt$
    and the tipping time $t_{\mt}$ is positive.
    \item
    $x(m)$ oscillates around the stable manifold before the tipping time and the tipping time $t_{\mt}$ is negative.
    Precisely, $x(m) > 0$ and $(-1)^m(x(m) - \sqrt{-t_m})>0$ for $m < \mt$
    and the tipping time $t_{\mt}$ is negative.
\end{enumerate}

This raises a question: when the ratio between $\Delta t$ and $\epsilon$ increases, how does the solution $\{x(m)\}_{m\in\bN}$ appear to change?  
To study this question, we solve the problem~(\ref{e:tipping}), that is,
the equation~(\ref{e:maina}) with the initial condition
$
t_0 = -1, \quad x_0=\sqrt{-t_0}+\epsilon
$
for various pairs of $(\epsilon,\Delta t)\in(0,0.5]\times(0,0.5]$. Then, for each pair 
$(\epsilon,\Delta t)$, we track the trajectory of the solution $\{x(m)\}_{m\in\bN}$ from $t_0$ to the tipping time $t_{\mt}$.
The results are summarized in Figure~\ref{fig:Tipping_Type} and Figure~\ref{fig:Tipping_Type3}.
\begin{figure}[b!]
    \centering
    \begin{subfigure}[b]{0.73\textwidth}
        \centering
        \includegraphics[width=\textwidth]{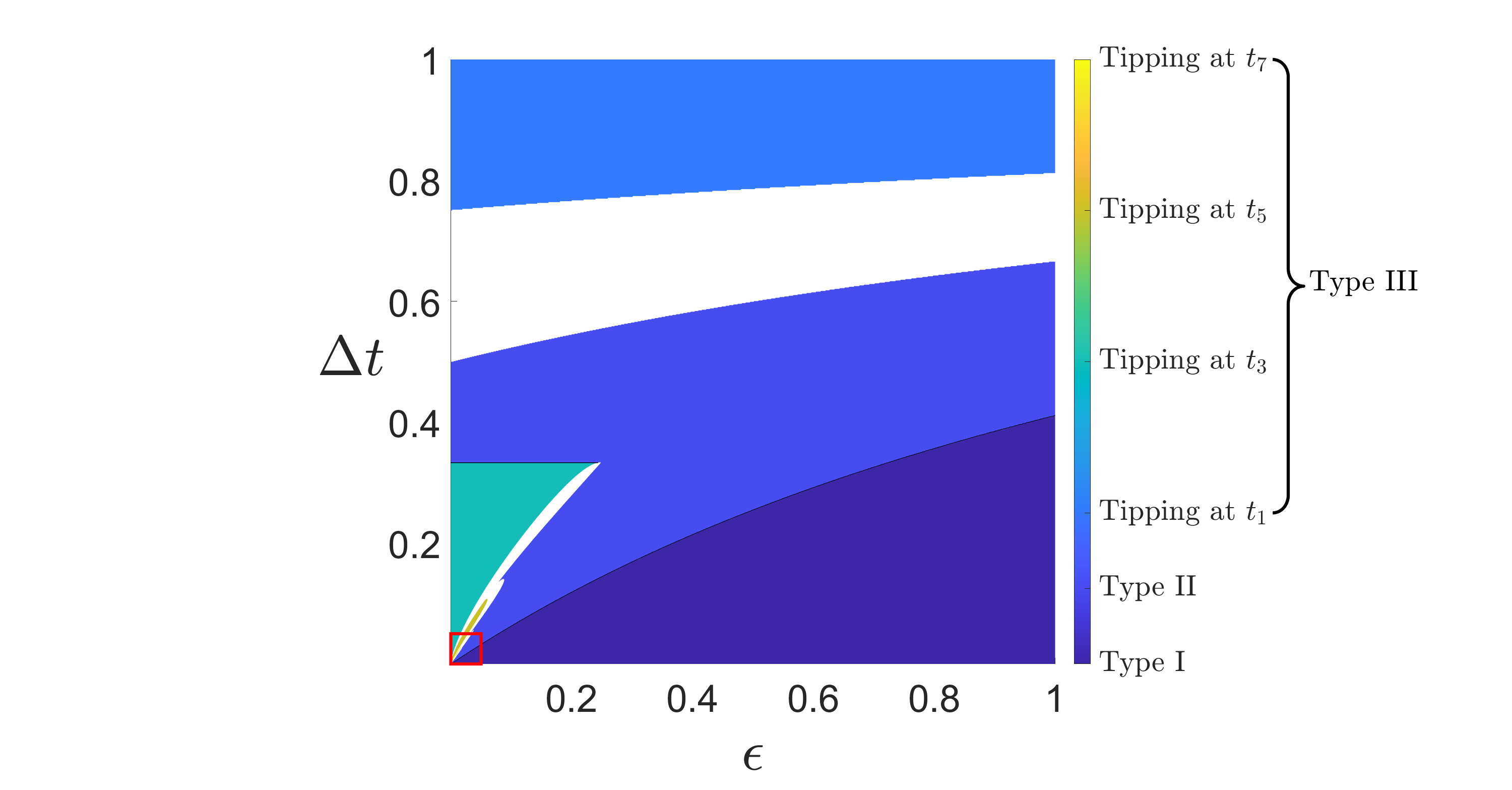}
        \caption{}
    \end{subfigure}
    \hfill
    \begin{subfigure}[b]{0.73\textwidth}
        \centering
        \includegraphics[width=\textwidth]{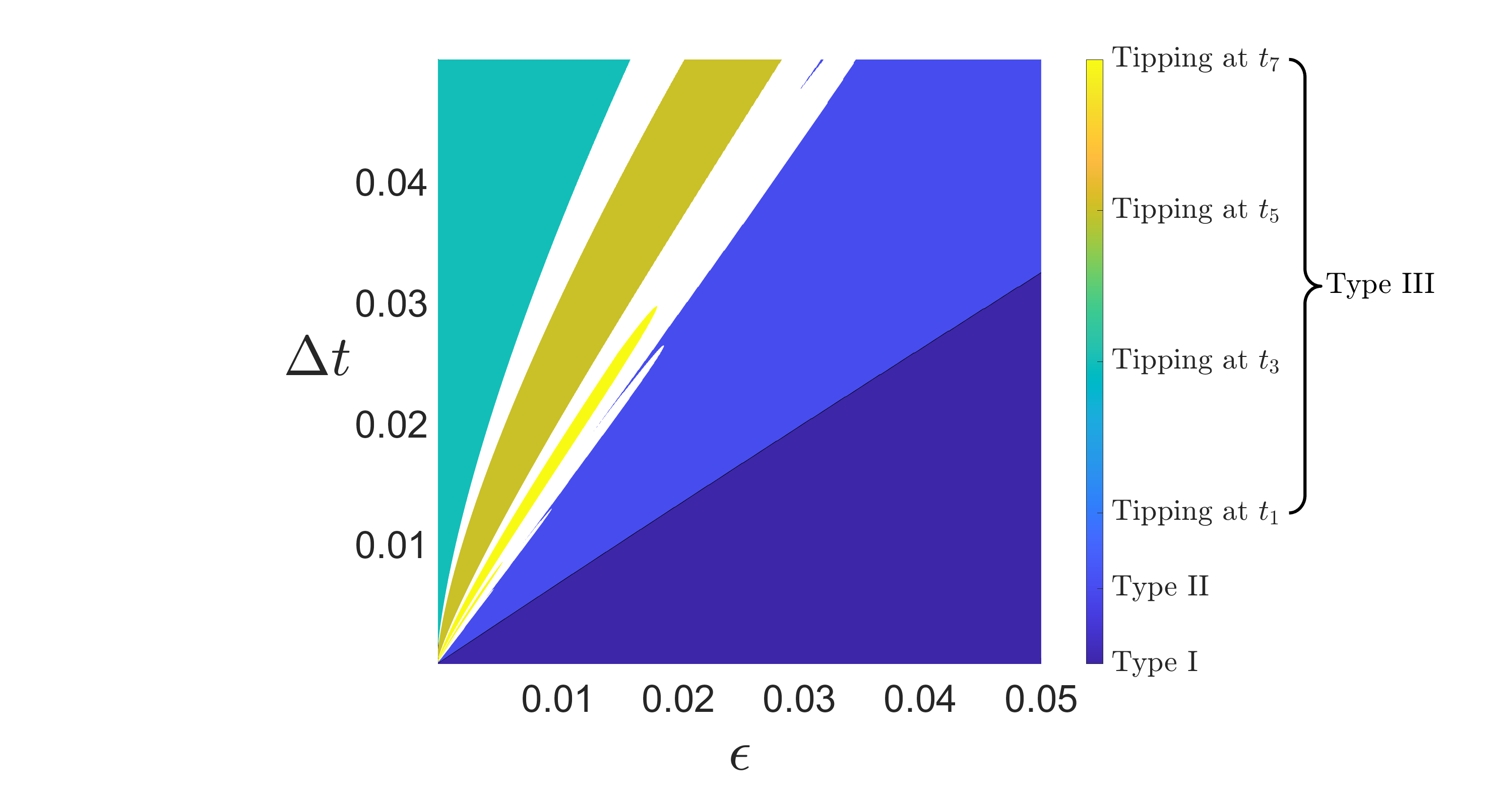}
        \caption{}
    \end{subfigure}
    \caption{\label{fig:Tipping_Type}
    Panel (a): The parameter regions I, II, and III in the $(\epsilon, \Delta t)$-plane corresponding to solution types (I), (II), and (III), respectively  (see texts).
    Panel (b): The magnification of the red square region in Panel (a).
    }
\end{figure}
\begin{figure}[ht!]
    \centering
    \includegraphics[width=\textwidth]{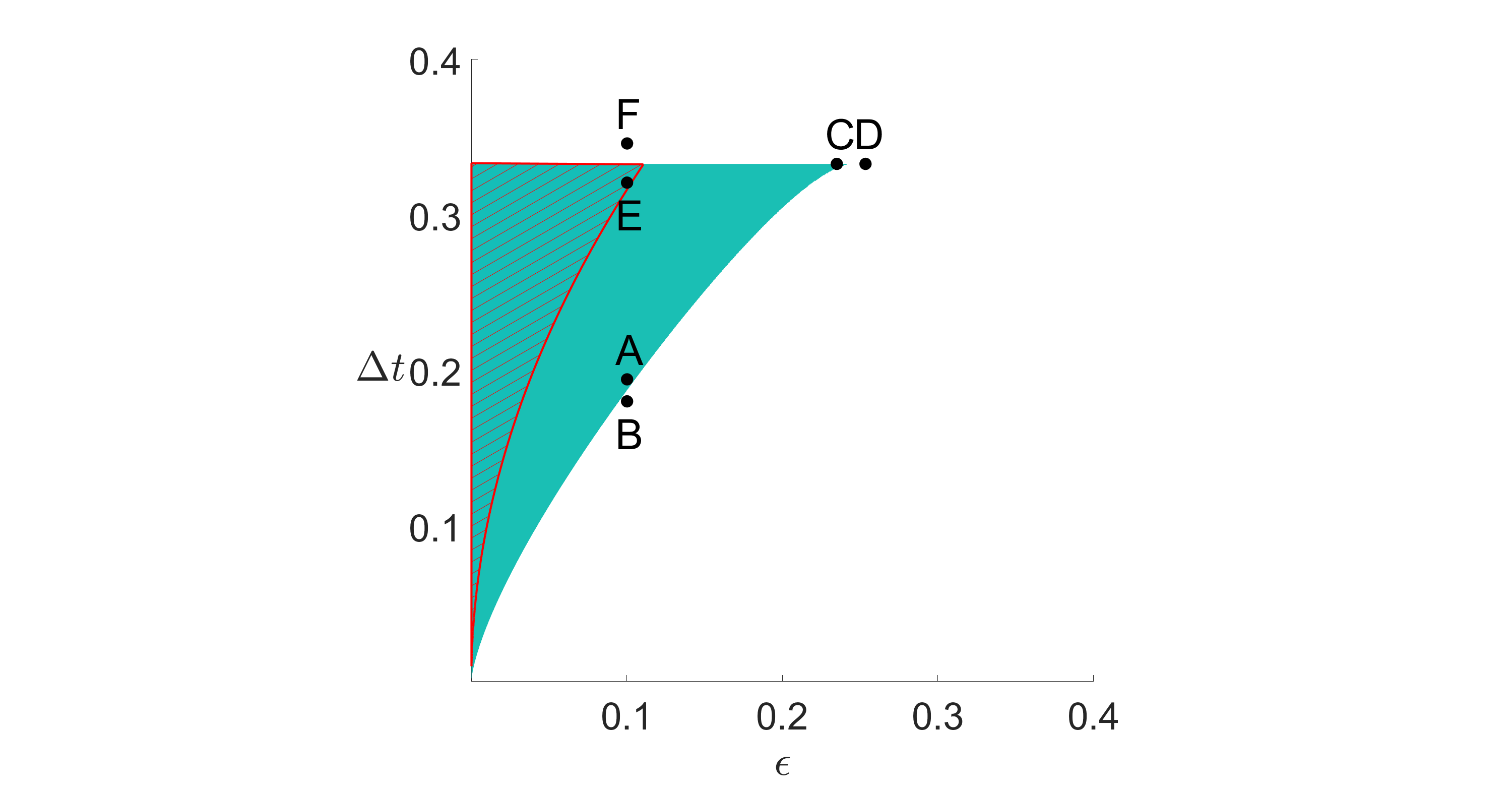}
    \caption{\label{fig:Tipping_Type3}
     The magnification of the green-blue region of Figure~\ref{fig:Tipping_Type}(a)
     corresponding to type (III) solutions with negative tipping time $t_{\mt}$ and $\mt=3$.
     The red hatched region corresponds to $\Delta t = \epsilon^b,~b\in(0,1/2)$, predicted by Lemma~\ref{l:tippig_at_t3}.}
\end{figure}
\begin{figure}[ht!]
    \centering
    \includegraphics[width=0.49\textwidth]{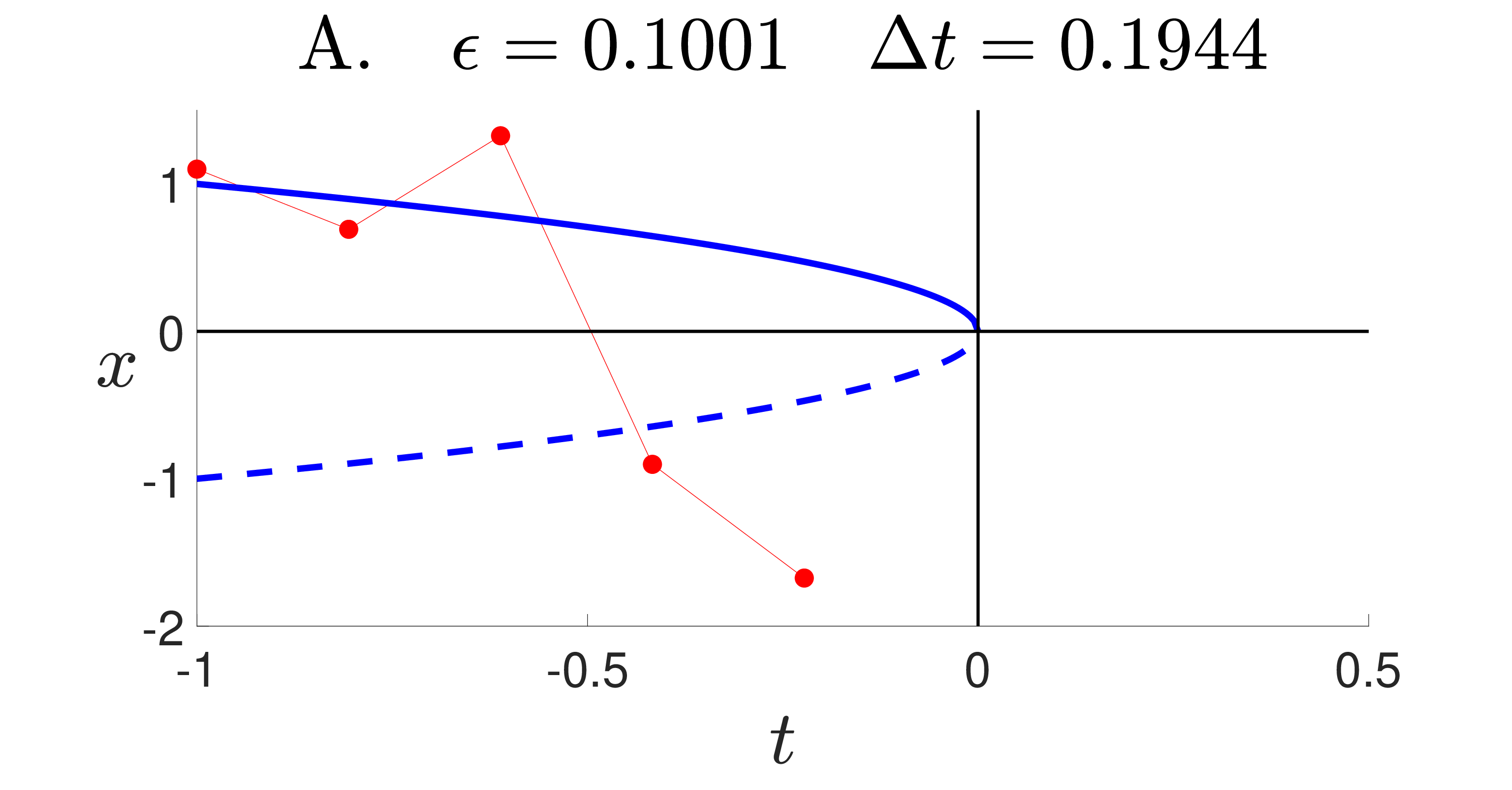}
    \includegraphics[width=0.49\textwidth]{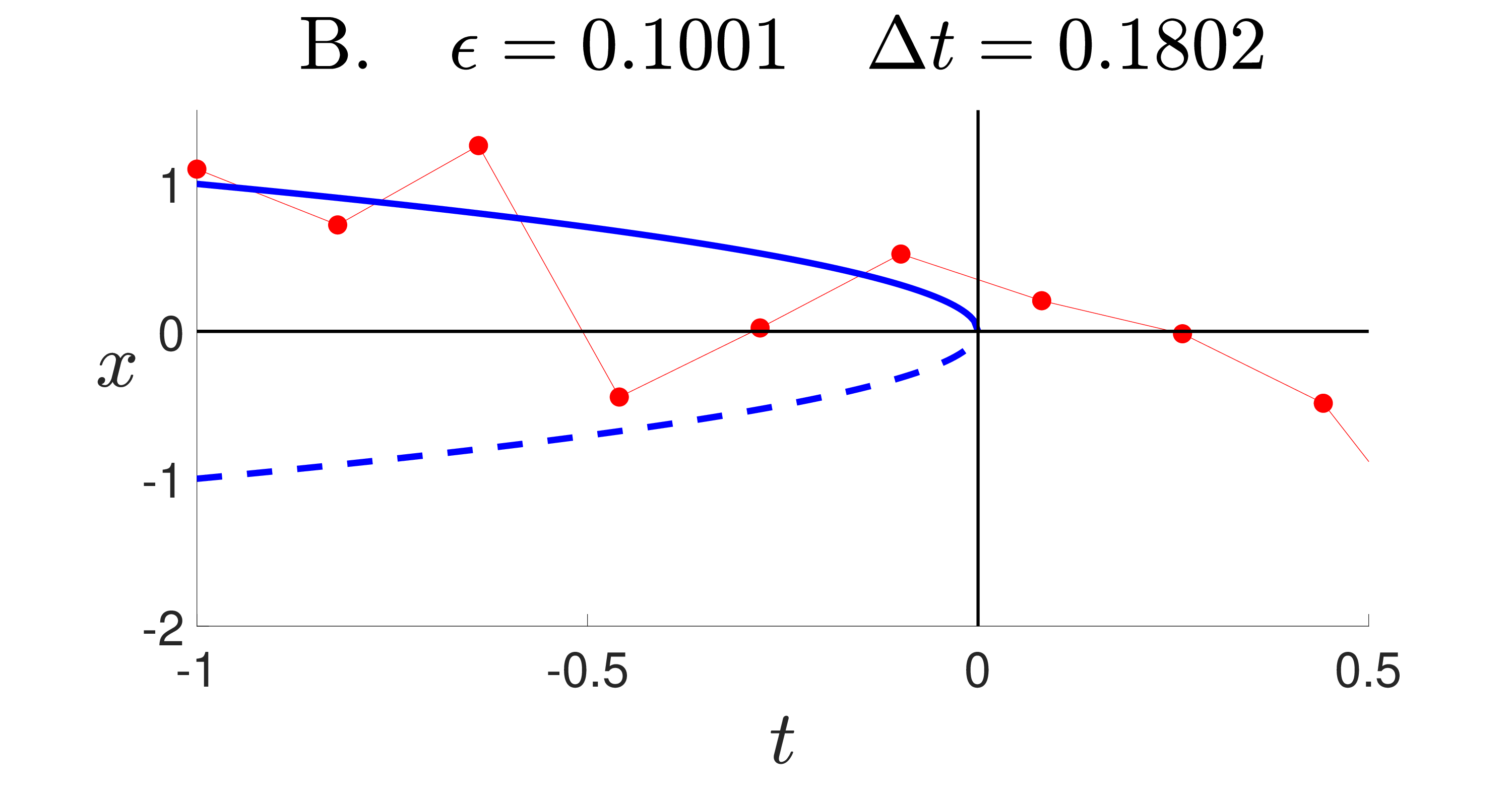}
    \includegraphics[width=0.49\textwidth]{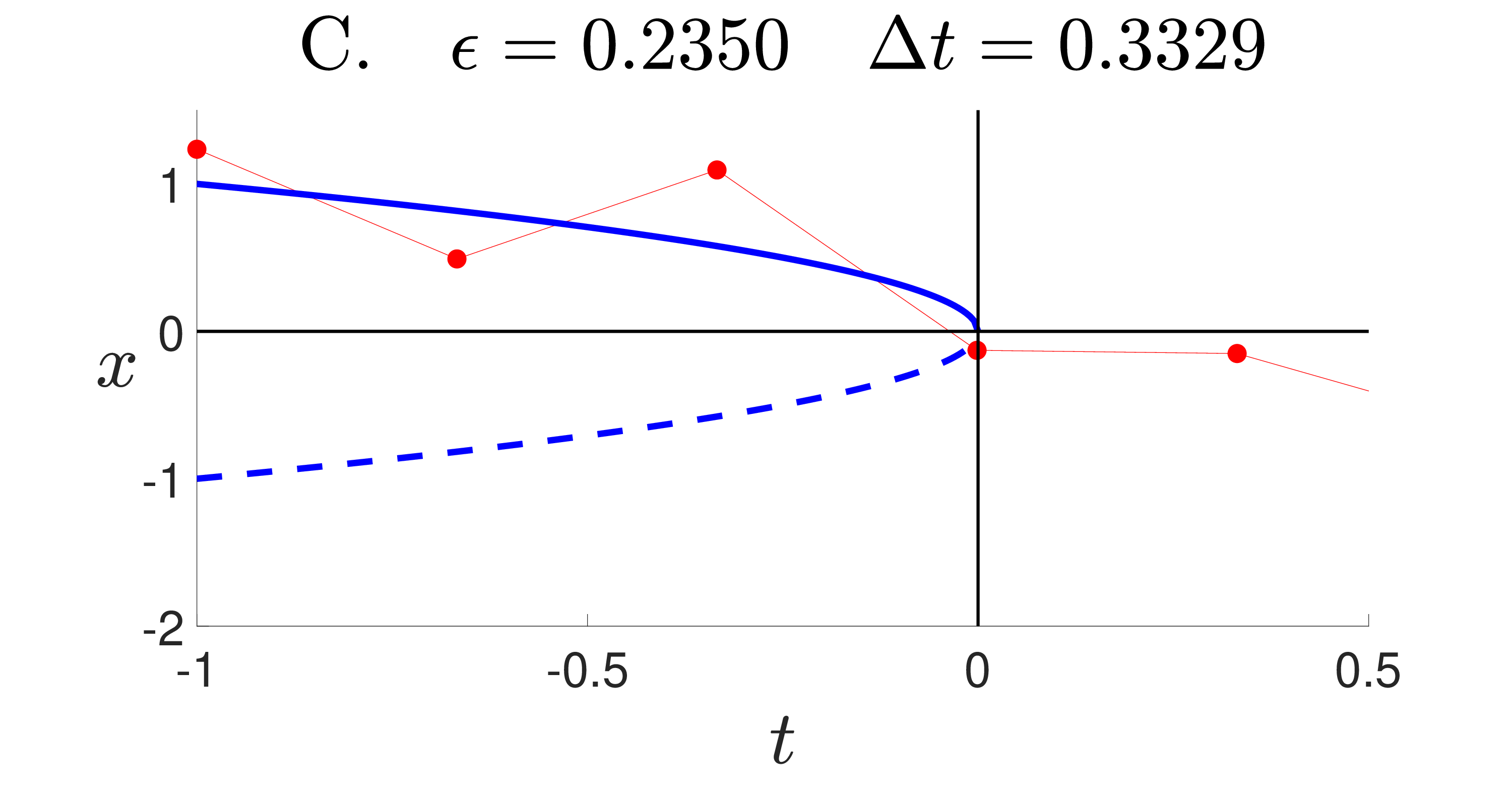}
    \includegraphics[width=0.49\textwidth]{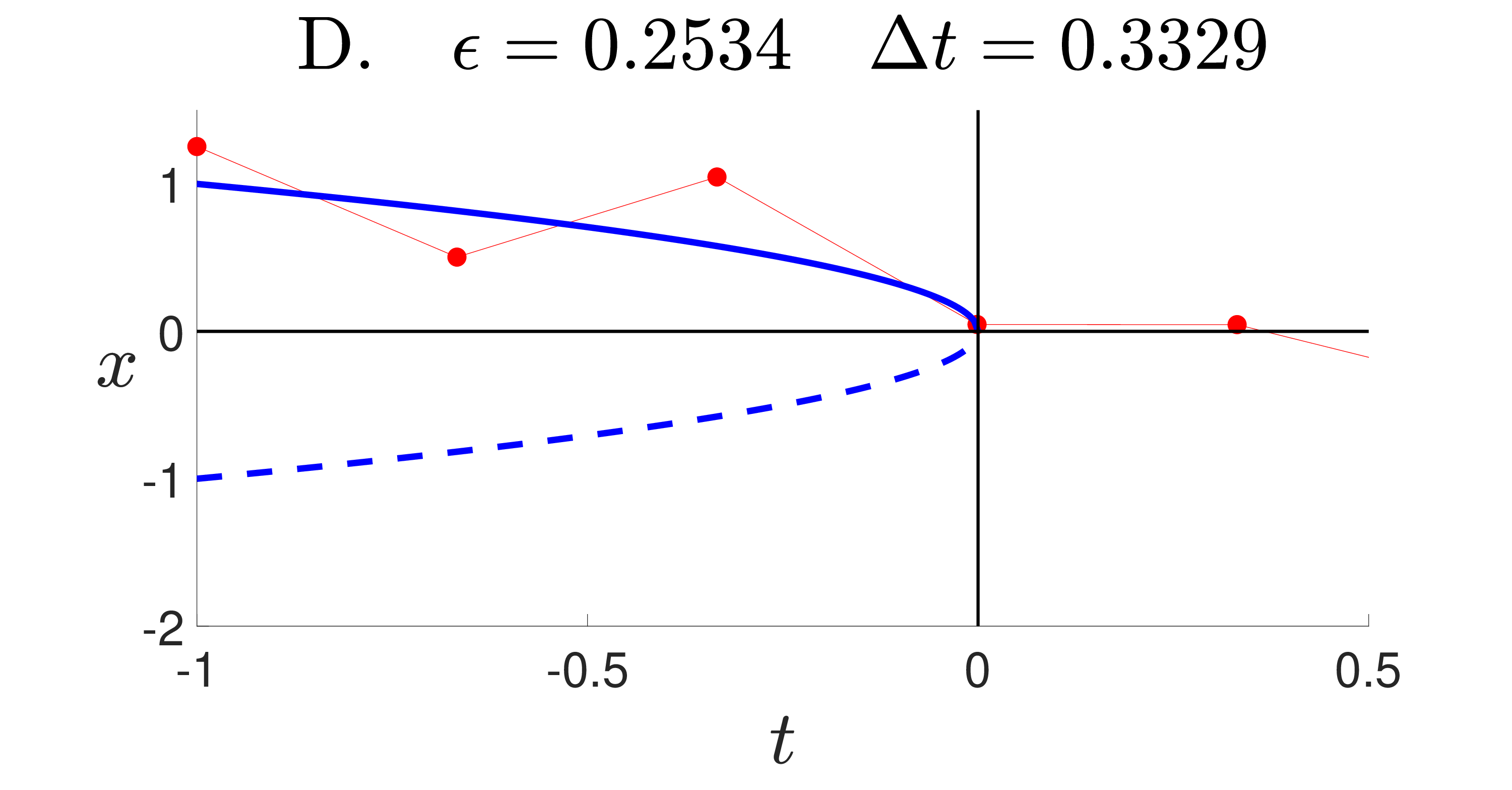}
    \includegraphics[width=0.49\textwidth]{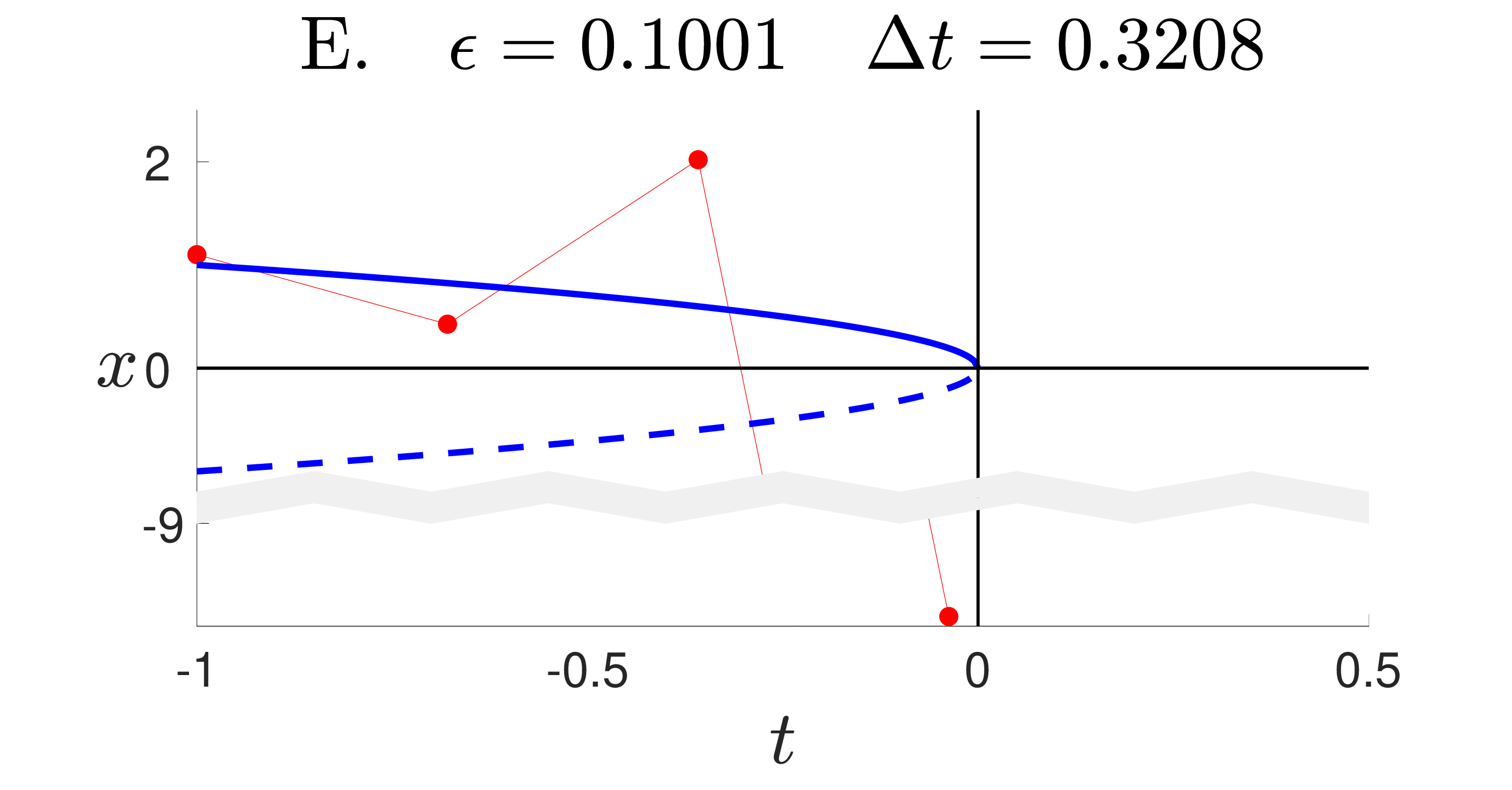}
    \includegraphics[width=0.49\textwidth]{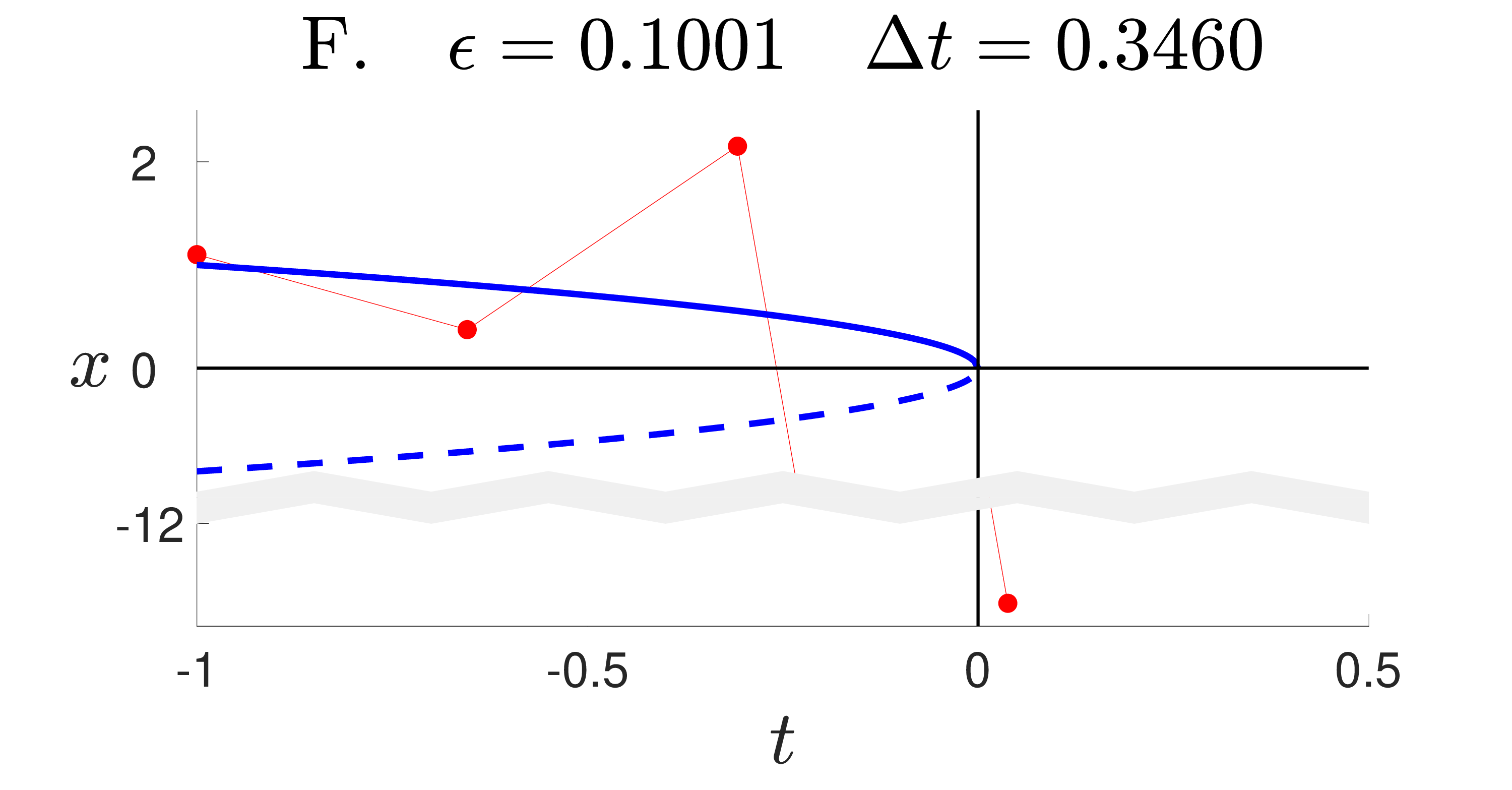}
    \caption{\label{fig:ratio_boundary}
    The time courses of solutions to problem~(\ref{e:main}) corresponding to the points A, B, C, D, E, and F, marked in the Figure~\ref{fig:Tipping_Type3}.
    }
\end{figure}

\subsection{Boundary of the parameter region with negative tipping $t_m$}

Now, let $\Omega_m$ be the region in the $(\epsilon, \Delta t)$-plane for which the corresponding solution of the problem~(\ref{e:tipping}) has negative tipping time $t_{\mt}$ and $\mt=m$.  
We would like to investigate the boundary of the region $\Omega_m$.
The following lemma gives the upper bound for the $\Delta t$-component of the region $\Omega_m$ when the tipping time is negative.
\begin{lemma}[\bf Upper bound of $\Delta t$ with $t_m<0$] \label{l:upperbound_tip}
    Let $\{ x(m) \}_{m\in\mathbb{N}}$ be the solution of problem~(\ref{e:main}) with initial condition $t_0=-1$. Suppose that $x(M)$ is the tipping point for some $t_M<0$. Then the time mesh size $\Delta t<1/M$.
\end{lemma}
\begin{proof}
    This can be proved by contradiction. Suppose that $\Delta t\geq 1/M$. Then the tipping time
    \begin{equation*}
        t_M=t_0+M\Delta t \geq -1+M\frac{1}{M}=0,
    \end{equation*}
    which contradicts to $t_M<0$.
    Therefore, the time mesh size $\Delta t$ is less than $1/M$.
\end{proof}

We remark that
as shown in Figure~\ref{fig:Tipping_Type3},  
the top boundary of the region $\Omega_m$ with negative tipping time at $t_3$ seems to be $\Delta t=1/3$.
This is consistent with that predicted by Lemma~\ref{l:upperbound_tip}.
On the other bound, the upper bound for the top boundary of $\Omega_m$ given in Lemma~\ref{l:upperbound_tip}, 
is overestimated, as indicated by Figure~\ref{fig:Tipping_Type}.

Next, we want to observe the boundaries of $\Omega_m$ in the $(\epsilon, \Delta t)$-plane for each $m$. For convenience, we will denote the top boundary of the $\Omega_m$ region as $\Gamma_m^+$ and the bottom boundary as $\Gamma_m^-$.
Take several points on each boundary $\Gamma_m^\pm$ and plot a log-log graph of $\Delta$ against $\epsilon$ at those boundary points.
As shown in Figure~\ref{fig:Tipping_Type_Bdy}(b), we can observe that the points on the boundaries satisfy the relation $\Delta t = C\epsilon^b$. Furthermore, $\Gamma_m^-$ and $\Gamma_{m+2}^+$ in Figure~\ref{fig:Tipping_Type_Bdy}(b) have similar slopes, meaning that $\Gamma_m^-$ and $\Gamma_{m+2}^+$ have approximately the same exponent with respect to $\epsilon$.
We numerically compute the coefficients $C$ and exponents $b$ corresponding to each boundary $\Gamma_m^\pm$ and display the results in Table~\ref{table:tippingtype_bdy}. The numerical results show that as $m$ increases, the exponent of $\epsilon$ increases towards 1.
We have the following conjecture about the exponents $b_m$ of $\epsilon$ for boundaries $\Gamma_m^-$, $\Gamma_{m+2}^+$:
\begin{equation*}
    b_m=1-\frac{1}{pm+q}
\end{equation*}
where $p,~q$ are constants. For the boundaries in Figure~\ref{fig:Tipping_Type_Bdy}, our estimated values of $p$ and $q$ are approximately 1.3543 and -1.0362, respectively. The estimated and numerical results for $b_m$ 
 are plotted in Figure~\ref{fig:Tipping_Type_Bdy_Estimate}.

\begin{figure}[ht!]
    \centering
    \begin{subfigure}[b]{0.87\textwidth}
        \centering
        \includegraphics[width=\textwidth]{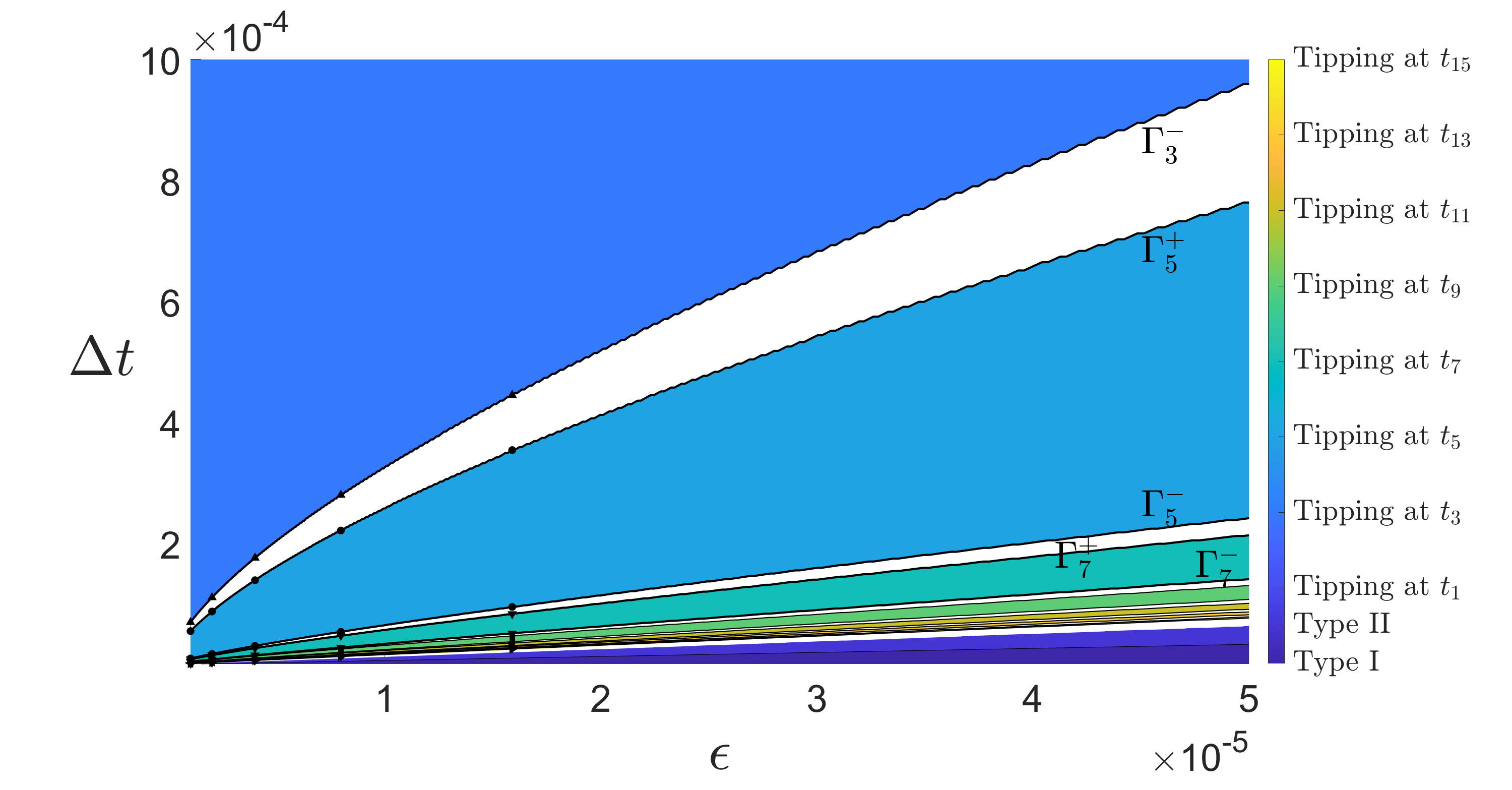}
        \caption{}
    \end{subfigure}
    \hfill
    \begin{subfigure}[b]{0.87\textwidth}
        \centering
        \includegraphics[width=\textwidth]{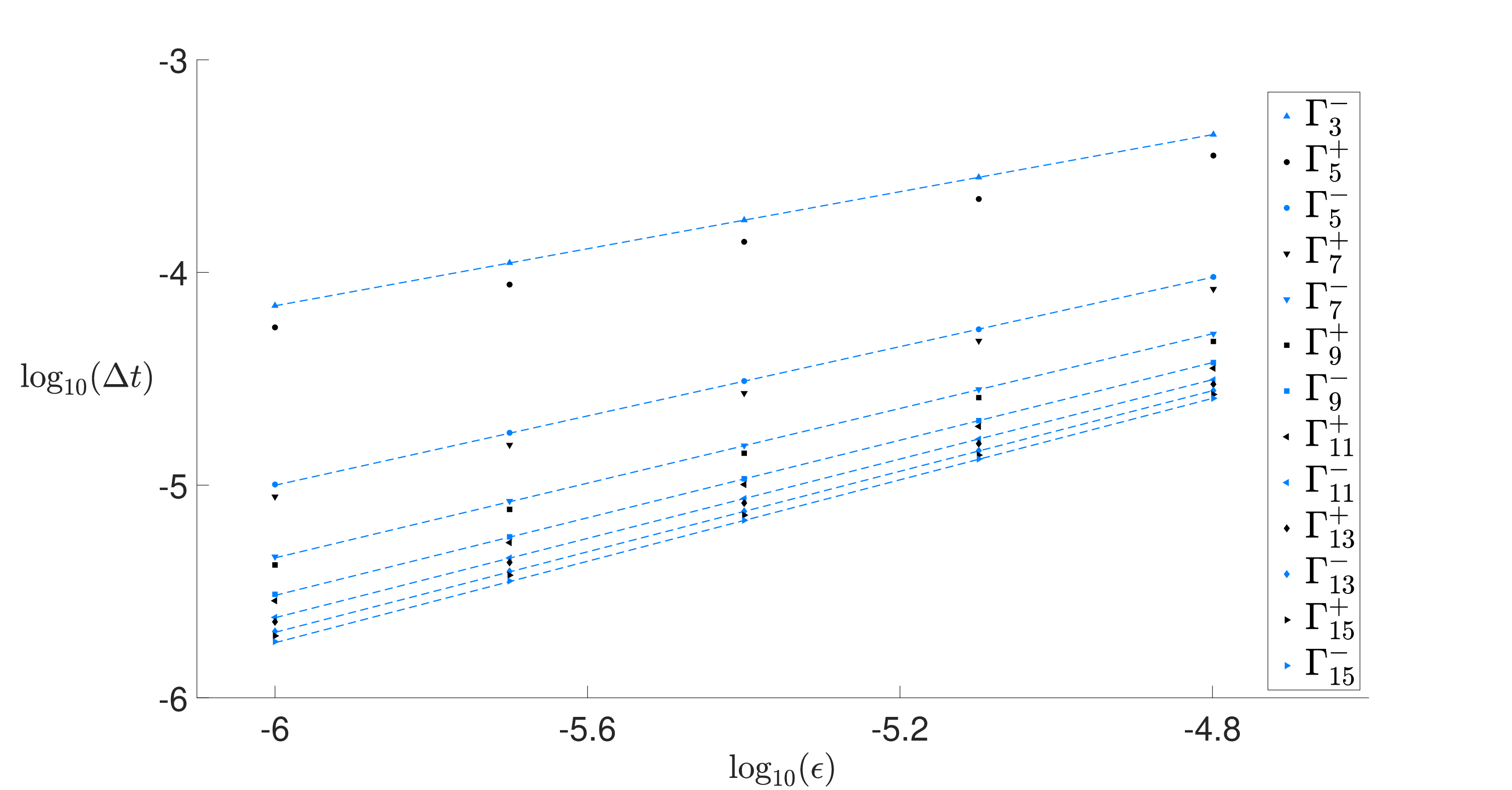}
        \caption{}
    \end{subfigure}
    \caption{\label{fig:Tipping_Type_Bdy}
    Panel (a): The region $\Omega_m$ and its  boundaries $\Gamma_m^\pm$ for $m=1, 3, \ldots, 15$.
    Panel (b): The log-log plot of $\Gamma_m^\pm$.  
    Marker points correspond to those in Panel (a).
    }
\end{figure}

\begin{table}[ht!]
    \centering
    \begin{tabular}{|c|c c c c c c|}
        \hline
        Boundary &  
        $\Gamma_3^-$ & $\Gamma_5^+$ & $\Gamma_5^-$ & $\Gamma_7^+$ & $\Gamma_7^-$ & $\Gamma_9^+$\\
        \hline
        $C$ &
        0.7362 & 0.5995 & 0.7729 & 0.7033 & 0.8439 & 0.7973\\
        $b$ & 
        0.6706 & 0.6730 & 0.8148 & 0.8178 & 0.8780 & 0.8805\\
        \hline
    \end{tabular}
    \\ \hspace*{\fill} \\
    \begin{tabular}{|c|c c c c c c c|}
        \hline
        Boundary &  
        $\Gamma_9^-$ & $\Gamma_{11}^+$ & $\Gamma_{11}^-$ & $\Gamma_{13}^+$ & $\Gamma_{13}^-$ & $\Gamma_{15}^+$ & $\Gamma_{15}^-$\\
        \hline
        C &
        0.9002 & 0.8593 & 0.9401 & 0.9005 & 0.9715 & 0.9103 & 1.0072\\
        b & 
        0.9121 & 0.9135 & 0.9328 & 0.9334 & 0.9466 & 0.9446 & 0.9573\\
        \hline
    \end{tabular}
    \caption{
    The coefficients $C$ and exponents $b$ corresponding to each boundary $\Gamma_m^\pm$, where all boundaries agree with the following function: $\Delta t = C\epsilon^b$.}
    \label{table:tippingtype_bdy}
\end{table}

\begin{figure}[ht!]
    \centering
    \includegraphics[width=0.87\textwidth]{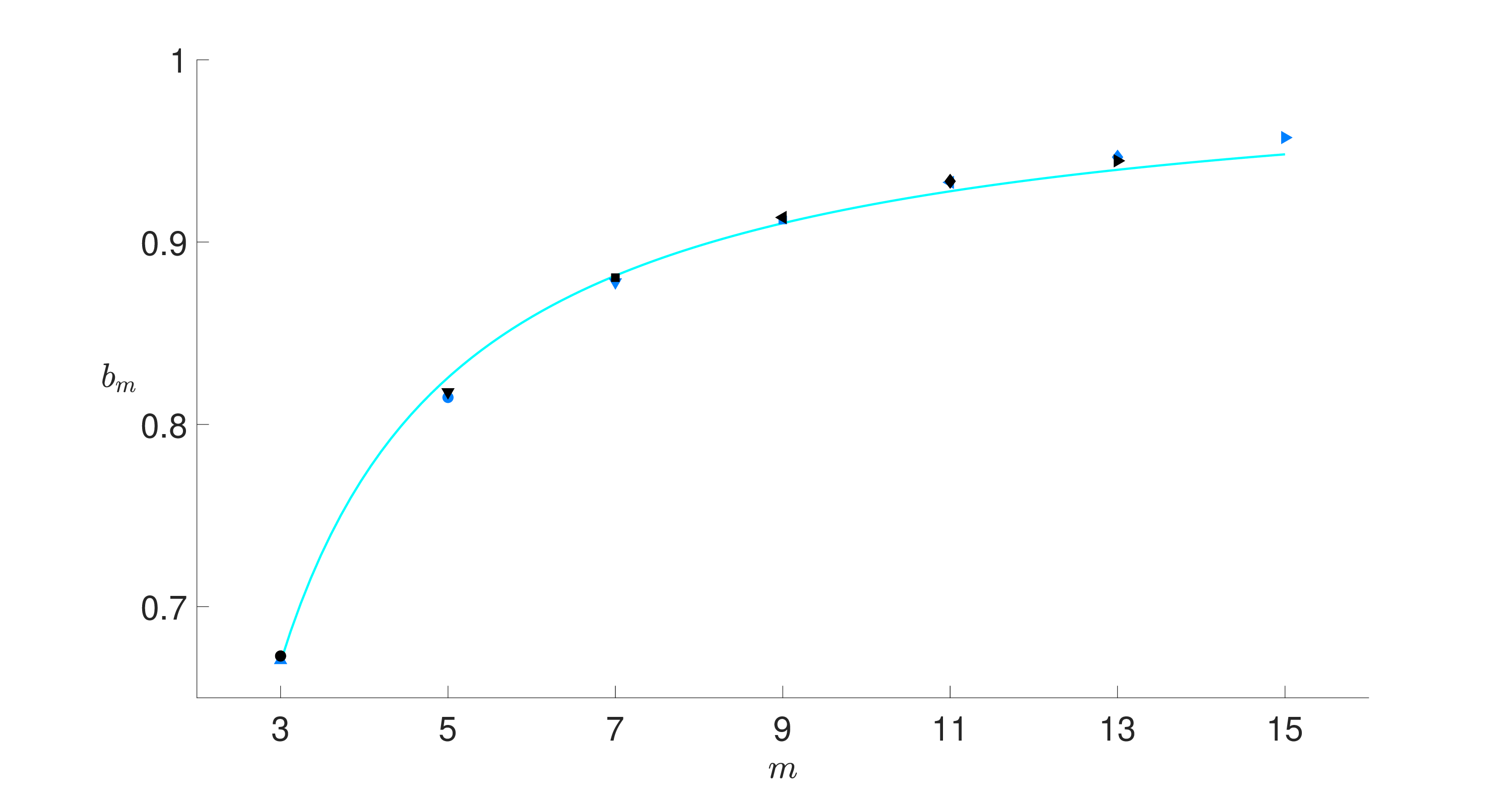}
    \caption{\label{fig:Tipping_Type_Bdy_Estimate}
    For each m, the estimated and numerical results for the exponent $b_m$ of $\epsilon$ in the functions for boundaries $\Gamma_m^-$ and $\Gamma_{m+2}^+$.
    The light blue curve represents the estimated values $1-1/(pm+q)$, while the other markers correspond to the numerical results for each boundary in Figure~\ref{fig:Tipping_Type_Bdy}(b).
    }
\end{figure}
\section{Application:}
\setcounter{equation}{0}
Bistable systems, serving as switches, constitute a fundamental element within electronic and optical devices, such as a memory element in an electronic system that needs to be repetitively switched on and off.
And they can be effectively described by a single dynamical equation~\cite{Jung:1990}.
Here, we consider the following one-dimensional model for a
switched bistable system:
\begin{equation}\label{e:application_temp10}
    \frac{dy}{d\tau} = y-\frac{y^3}{3}-p(\tau),
\end{equation}
where $p$ is the dynamic control parameter given by $p(\tau)=\epsilon\tau$ and $\epsilon\in(0,1)$ is a sweep rate parameter. To understand the dynamic behavior of system~(\ref{e:application_temp10}), we can observe its autonomous counterpart system~(\ref{e:application_temp40}):
\begin{equation}\label{e:application_temp40}
    \frac{dy}{d\tau}=y-\frac{y^3}{3}-p
\end{equation}
Note that for each $p\in\mathbb{R}$, we can solve $y^*-(y^*)^3/3-p=0$ to find equilibrium $y^*$ in system~(\ref{e:application_temp40}), where $y^*>1$ and $y^*<-1$ are stable, while $-1<y^*<1$ is unstable. Moreover, if the initial value is greater than the stable equilibrium above, as $p$ is increased, the solution of system~(\ref{e:application_temp40}) tracks the stable equilibrium above, and then goes to the stable equilibrium below once the bifurcation parameter $p$ is crossed the bifurcation point $2/3$ from below. However, when viewed as a function of $p = \epsilon\tau$, the solutions of system~(\ref{e:application_temp10}) do not immediately switch to the stable equilibrium below when $p$ is increased through the bifurcation point $p=2/3$. Instead, it exhibits a bifurcation delay.

Use forward Euler scheme to discretize the equation~(\ref{e:application_temp10}), we obtain the following discrete dynamical system for $\{ Y(m) \}$:
\begin{equation}\label{e:application_temp20}
    Y(m+1)=Y(m)+\big(Y(m)-\frac{Y(m)^3}{3}\big)\Delta \tau-\epsilon\tau_m\Delta \tau, \quad m\in\mathbb{N}\cup\{ 0 \}
\end{equation}
where the time mesh size $\Delta \tau$ is a positive constant and $\tau_m =\tau_0+m\Delta\tau$. We will observe the dynamical behavior of equation~(\ref{e:application_temp20}). Use the scaling
\begin{equation}\label{e:application_temp30}
    y(m)=Y(m),\quad t_m=\epsilon\tau_m,\quad \Delta t=\epsilon\Delta tau,
\end{equation}
the equation is converted into
\begin{equation*}
    y(m+1)=\frac{\Delta t}{\epsilon}(y(m)-\frac{y(m)^3}{3})+y(m)-\frac{\Delta t}{\epsilon}t_m.
\end{equation*}
Note that the scaling time variable $t$ is exactly the slowly varying bifurcation variable $p$.

Now, consider the initial value problem
\begin{subnumcases} 
{\label{e:application}}
     y(m+1)=\frac{\Delta t}{\epsilon}(y(m)-\frac{y(m)^3}{3})+y(m)-\frac{\Delta t}{\epsilon}t_m, \label{e:applicationa} &\\
     y(0)= y_0+\epsilon, \label{e:applicationb} &
\end{subnumcases}
where $y_0\approx 2.1038$ satisfies $y_0-y_0^3/3-t_0=0$ for $t_0=-1$.


\begin{figure}[ht!]
    \centering
    \begin{subfigure}[b]{0.49\textwidth}
        \centering
        \includegraphics[width=\textwidth]{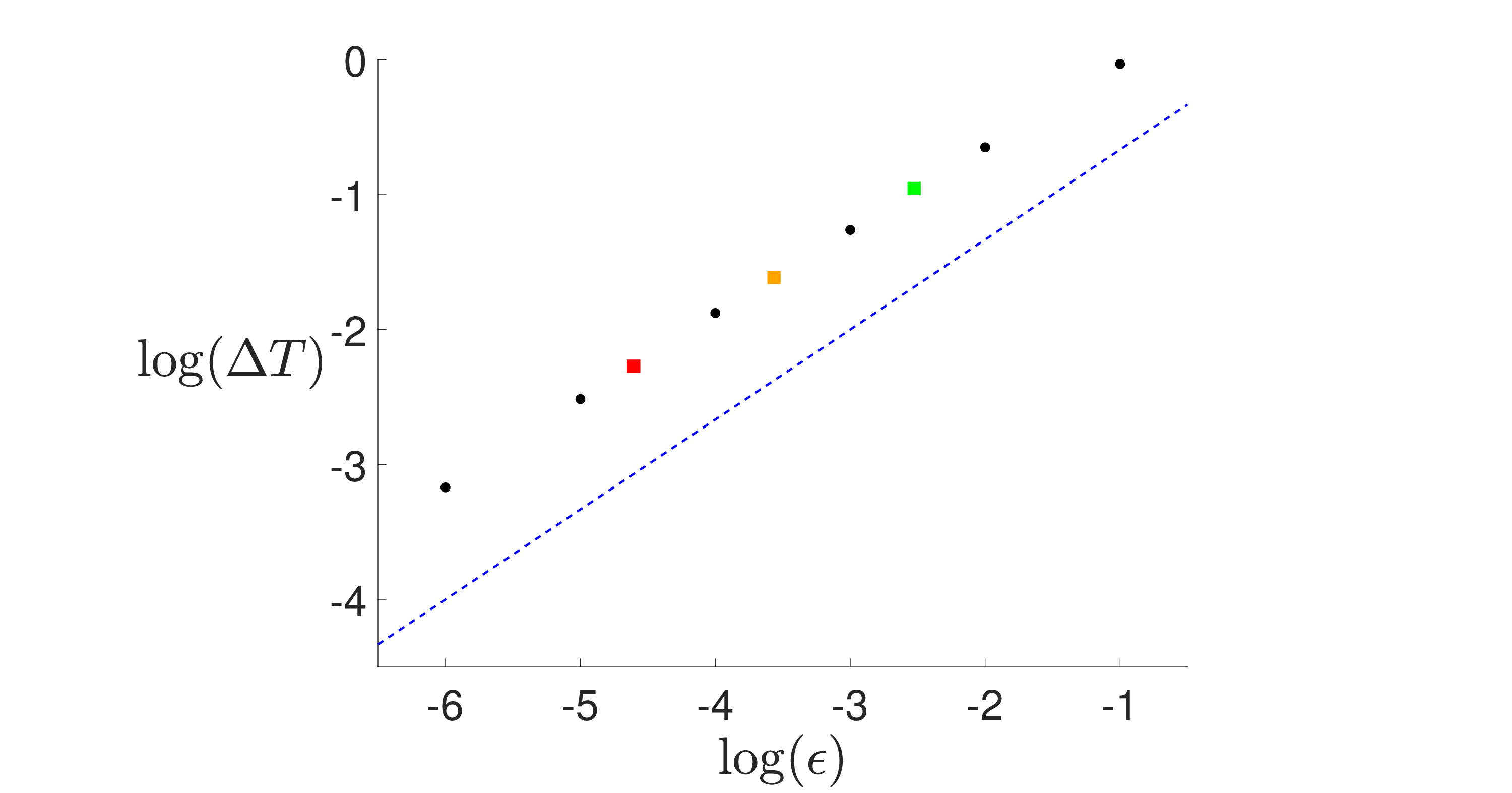}
        \caption{\label{fig:bistable_time_order}}
    \end{subfigure}
    \hfill
    \begin{subfigure}[b]{0.49\textwidth}
        \centering
        \includegraphics[width=\textwidth]{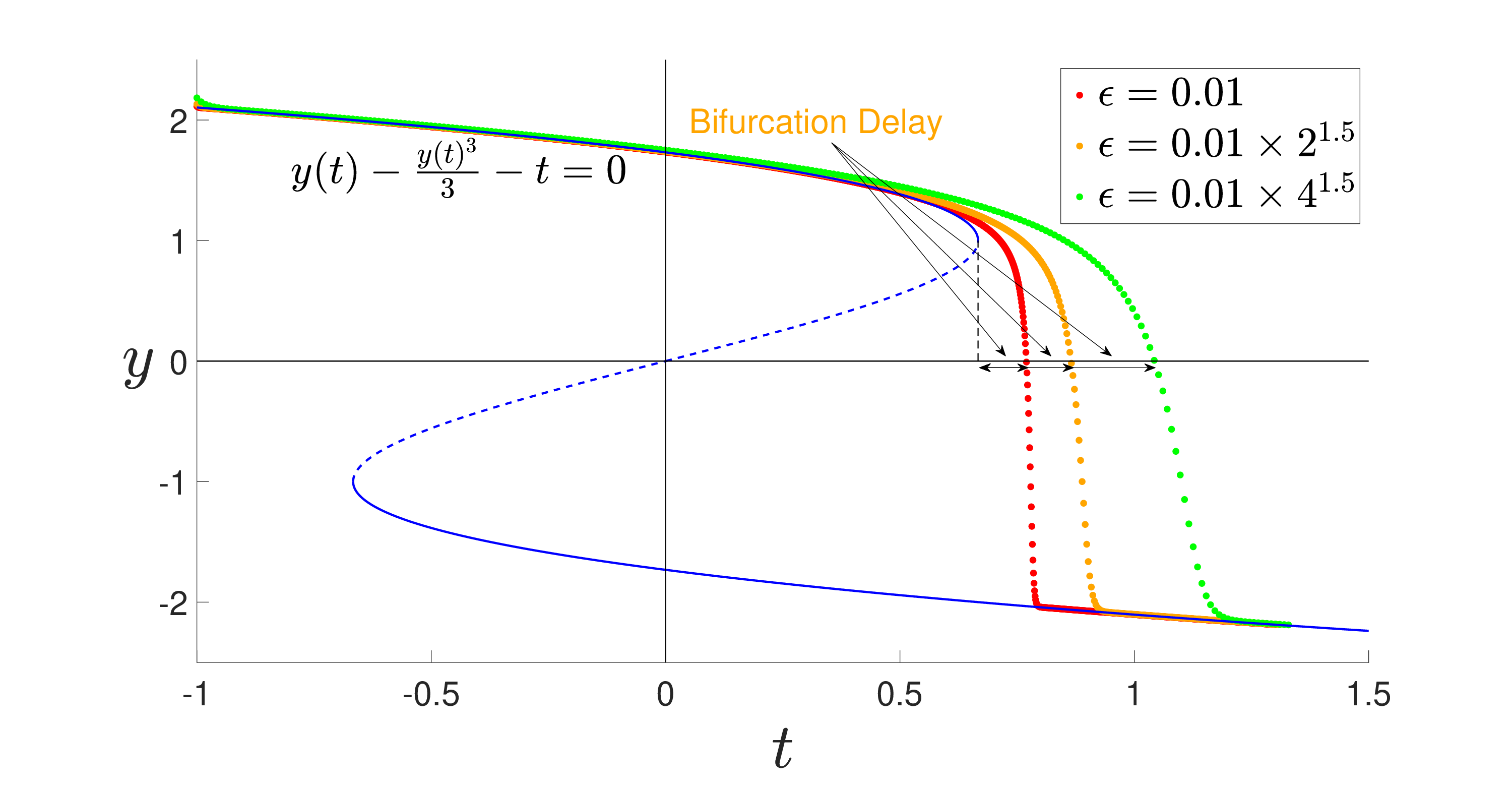}
        \caption{\label{fig:bistable_multi_sol}}
    \end{subfigure}
    \caption{
    (a). The dependence of bifurcation delay time $\Delta T$($=t_{\mt}-2/3$) on the sweep rate parameter $\epsilon$: $\epsilon=0.01$ (red marker), $\epsilon=0.01\times 2^{1.5}$ (orange marker), $\epsilon=0.01\times 4^{1.5}$ (green marker). The blue dashed line is $\log(t)= (2/3) \log(\epsilon)$. 
    (b). Solution profiles corresponding to the red, the orange, and the green markers in panel (a), respectively. 
    Here $\Delta t/\epsilon$ is fixed to be $\ln(2)/6$.
    }
\end{figure}




\end{document}